\documentclass[11pt, reqno]{amsart}
\usepackage{microtype}
\usepackage[T1]{fontenc}
\usepackage{amsmath, amscd, amsthm} 
\usepackage{amssymb, amsfonts} 
\usepackage{stmaryrd}
\usepackage{enumerate}
\usepackage[svgnames]{xcolor} 
\usepackage{color}
\usepackage[margin=1.5in]{geometry}
\usepackage{mathrsfs}
\usepackage{booktabs}
\usepackage{aliascnt}
\usepackage{mathrsfs}
\usepackage{hyperref}
 \definecolor{dark-red}{rgb}{0.4,0.15,0.15}
 
 \newcommand{\RR}{\mathcal{R}}

\newcommand{\LL}{\mathcal{L}}

 \usepackage[all]{xy}
 \xyoption{tips}
\SelectTips{cm}{10}
 
\hypersetup{
    colorlinks, linkcolor=dark-red,
    citecolor=DarkBlue, urlcolor=MediumBlue
}
\usepackage{mathrsfs}
\usepackage[
textwidth=3cm,
textsize=small,
colorinlistoftodos]
{todonotes}

\usepackage{cancel}
\usepackage{float}

\usepackage{tikz}
\usepackage{tikz-cd}
\usetikzlibrary{positioning,arrows}
\usetikzlibrary{shapes.geometric, calc}
\usetikzlibrary{decorations.pathreplacing}

\newcommand*\circled[1]{\tikz[baseline=(char.base)]{
            \node[shape=circle,draw,inner sep=2pt,fill=white] (char) {#1};}}

  \def\centerarc[#1](#2)(#3:#4:#5)% Syntax: [draw options] (center) (initial angle:final angle:radius)
    { \filldraw[#1] (0,0) -- ($(#2)+({#5*cos(#3)},{#5*sin(#3)})$) arc (#3:#4:#5) -- (0,0); }
    
\definecolor{gRed}{HTML}{ff5100}
\definecolor{gBlue}{HTML}{2b83ba}

\newcommand{\Set}{\operatorname{Set}}
\newcommand{\Inj}{\operatorname{Inj}}
\newcommand{\Surj}{\operatorname{Surj}}

\newcommand{\Sub}{\operatorname{Sub}}

\newcommand{\id}{\operatorname{id}}

\newcommand{\ob}{\operatorname{ob}}

\usepackage[T1]{fontenc}

\numberwithin{equation}{section} %Fiddles with numbering system of the following.

\theoremstyle{plain}

\newaliascnt{theorem}{equation}  
\newtheorem{theorem}[theorem]{Theorem}  
\aliascntresetthe{theorem}

 \theoremstyle{definition}

\newaliascnt{prop}{equation}  
\newtheorem{prop}[prop]{Proposition}
\aliascntresetthe{prop}

\newaliascnt{lemma}{equation}  
\newtheorem{lemma}[lemma]{Lemma}
\aliascntresetthe{lemma}

\newaliascnt{corollary}{equation}  
\newtheorem{corollary}[corollary]{Corollary}
\aliascntresetthe{corollary}

\newaliascnt{claim}{equation}  

\aliascntresetthe{claim}

\newaliascnt{conjecture}{equation}  

\aliascntresetthe{conjecture}

\newaliascnt{question}{equation}  

\aliascntresetthe{question}

\newaliascnt{defn}{equation}  
\newtheorem{defn}[defn]{Definition}
\aliascntresetthe{defn}

\newaliascnt{example}{equation}  
\newtheorem{example}[example]{Example}
\aliascntresetthe{example}

\theoremstyle{remark}

\newaliascnt{remark}{equation}  
\newtheorem{remark}[remark]{Remark}
\aliascntresetthe{remark}

\newaliascnt{convention}{equation}  

\aliascntresetthe{convention}

\newaliascnt{construction}{equation}  
\newtheorem{construction}[construction]{Construction}
\aliascntresetthe{construction}

\theoremstyle{plain}
\newtheorem*{mainthm}{Theorem}

\begin{document}
\title{Composition closed premodel structures and the Kreweras lattice}

\author{Scott Balchin}
\address{Max Planck Institute for Mathematics}
\email{balchin@mpim-bonn.mpg.de}

\author{Ethan MacBrough}
\address{Reed College}
\email{emacbrough@reed.edu}

\author{Kyle Ormsby}
\address{Reed College / University of Washington}
\email{ormsbyk@reed.edu \textnormal{/} ormsbyk@uw.edu}

\begin{abstract}
We investigate the rich combinatorial structure of premodel structures on finite lattices whose weak equivalences are closed under composition. We prove that there is a natural refinement of the inclusion order of weak factorization systems so that the intervals detect these composition closed premodel structures. In the case that the lattice in question is a finite total order, this natural order retrieves the Kreweras lattice of noncrossing partitions as a refinement of the Tamari lattice, and model structures can be identified with certain tricolored trees.
\end{abstract}

\maketitle

\vspace{-8mm}

\setcounter{tocdepth}{1}
\tableofcontents

\vspace{-6mm}

\section{Introduction}

A weak factorization system on a category $\mathcal{C}$ consists of a pair of classes of morphisms $(\mathcal{L}, \mathcal{R})$ satisfying lifting and factorization properties, and they make up the backbone of categorical homotopy theory.\footnote{Readers unfamiliar with categories, weak factorization systems, and model structures may skip to the start of \autoref{sec:prelim} for a brief review and references.} Given two such weak factorization systems $(\mathcal{L}, \mathcal{R})$ and $(\mathcal{L}' ,\mathcal{R}')$, there is an obvious ordering given by inclusion of the right class (equivalently by reverse inclusion of the left class). Such an ordered pair was called a \emph{premodel structure} by Barton~\cite{barton}. For a premodel structure there is a natural notion of weak equivalence, namely, we define $\mathcal{W} = \RR \circ \LL'$. If the class of morphisms $\mathcal{W}$ are closed under the two-out-of-three property, then this is exactly the data of a model structure on $\mathcal{C}$ in the sense of Quillen~\cite{quillen}.

If the category $\mathcal{C}$ happens to be finite lattice $L$, then the collection of weak factorization systems on $L$ is itself a finite lattice. As such, the premodel structures can be described as intervals in this lattice, in particular, the collection of premodel structures on a fixed $L$ form a finite set. Among these intervals, one may hope to be able to identify the model structures. It was this viewpoint that was taken by the first and third authors along with Osorno and Roitzheim in~\cite{boor}, where this was done for the finite total order $[n] = \{0<1<\cdots <n\}$. Such investigations were dubbed as \emph{homotopical combinatorics} in loc.~cit.

A key input to this was the identification of the weak factorization systems for $[n]$, which was achieved by Balchin--Barnes--Roitzheim~\cite{bbr} in tandem with work of Franchere--Ormsby--Osorno--Qin--Waugh~\cite{fooqw} under the guise of \emph{transfer systems} via a surprising link to equivariant homotopy theory for the group $C_{p^n}$. Namely, homotopy classes of $N_\infty$ operads for $C_{p^n}$ in the sense of~\cite{blumberghill} are in bijection with weak factorization systems on $\Sub(C_{p^n}) \cong [n]$. From these results, it is know that the lattice of weak factorization systems for $[n]$ is in bijection with the Tamari lattice (of order $n+1$), and as such the premodel structures are in bijection with intervals of the Tamari lattice. Among these intervals, an enumeration was given for the number of model structures.

It was left as an open question in~\cite[\S 6]{boor} to explicitly identify the model structures among the Tamari intervals. In this paper we provide an answer to this question. To do so, we investigate a structure which lives in between the premodel structures and model structures. Instead of asking for the weak equivalences to satisfy the 2-out-of-3 property, we instead ask them to only be closed under composition. It turns out that this collection of \emph{composition closed premodel structures} enjoys many excellent structural properties; in particular, we can identify them as intervals arising from a refined ordering on the collection of weak factorization systems.

\begin{mainthm}[\autoref{sec:ccclosedaslattice}]
Let $L$ be a finite lattice. Then there is a refinement of the ordering on weak factorization systems on $L$ such that the intervals are exactly the composition closed premodel structures. Moreover, the set of weak factorization systems equipped with this ordering is a finite lattice.
\end{mainthm}

This result suggests that the collection of composition closed premodel structures may provide a sensible middle ground between premodel structures and model structures. We highlight that even though there is a further refinement of the ordering on weak factorization systems such that the intervals are the model structures, the resulting poset is, not even in the simplest non-trivial cases, a lattice (c.f., \autoref{sec:refinedformodel}). As such, the collection of model structures on a fixed lattice is poorly behaved.

In \autoref{sec:caseoffinitetotal} we explore the implications of the above theorem for finite total orders. Recall that we know that the ordering on weak factorization systems retrieves the Tamari lattice. By a very explicit bijection we can consider weak factorization systems on $[n]$ equivalently as noncrossing partitions on $[n]$. The collection of noncrossing partitions under refinement of partitions gives us the \emph{Kreweras lattice}, a well known, and ubiquitous, refinement of the Tamari lattice. We are now in a position to state the main theorem of \autoref{sec:caseoffinitetotal}.

\begin{mainthm}[\autoref{thm:main}]
Let $n \geqslant 0$. Then the composition closed premodel structures on the lattice $[n]$ are in bijection with intervals of the Kreweras lattice.
\end{mainthm}

\begin{remark}
An anonymous referee points out the affinity between this theorem and the work of Ch\^atel--Pons \cite{chatelpons} and Rognerud \cite{rognerud}. Indeed, Rognerud utilizes the \emph{interval posets} of Ch\^atel--Pons to introduce \emph{exceptional intervals} of the Tamari lattice. These are in bijection with Kreweras intervals and thus, by \autoref{thm:main}, in bijection with composition closed premodel structures on finite total orders. Deeper analysis of the link between interval posets and abstract homotopy theory should prove fruitful.
\end{remark}

Equipped with an explicit description of the composition closed premodel structures for $[n]$, in \autoref{sec:stacks} we return to the question raised in~\cite{boor} regarding identifying the model structures among the Tamari intervals. Building on work of Bernardi--Bonichon~\cite{intervals}, we are able to construct a bijection between tricolored trees and composition closed premodel structures on $[n]$. We then provide an easy-to-check criterion for if a given composition closed premodel structure is a model structure.

\begin{mainthm}[\autoref{mainthm}]
There is an explicit bijection between certain tricolored trees and Kreweras pairs under which a tricolored tree represents a model structure if and only if it does not have a red branch descended from any non-red branch.
\end{mainthm}

This theorem is intriguing for several reasons. Firstly, the Kreweras lattice appears in many disparate areas of mathematics. Perhaps most interestingly for this paper is the fact that the thick subcategories of the Abelian category $\mathrm{Rep}(A_n)$ of all finite dimensional $k$-linear
representations of the quiver $A_n$ for a field $k$ under inclusion is the Kreweras lattice~\cite{ingallsthomas, krause_string}. The investigations here suggest that there may be nontrivial substance to understanding an explicit interpretation of the lattice intervals in this setting, especially those that are in bijection to the model structures.

Finally, returning to the setting of equivariant homotopy theory and $N_\infty$ operads, Blumberg--Hill~\cite{MR4327103} introduce the notion of \emph{compatible pairs} of $N_\infty$ operads, which encode \emph{bi-incomplete Tambara functors}. One can prove that the number of compatible pairs of $N_\infty$ operads for $C_{p^n}$ is exactly the number of Kreweras intervals. In particular, this suggests the existence of a (non-trivial) bijection between compatible pairs of $N_\infty$ operads and composition closed premodel structures on the subgroup lattice of $C_{p^n}$. We discuss this further in \autoref{rem:compatible}.

\subsection*{Acknowledgements}
The authors thank the anonymous referees for helpful comments and suggestions. The first author would like to thank the Max Planck Institute for Mathematics for its hospitality, and was partially supported by the European Research Council (ERC) under Horizon Europe (grant No. 101042990). The second author thanks Coil Technologies for their generous donation to fund his tuition, which enabled him to conduct this research. The third author's work was supported by the National Science Foundation under Grant No.~DMS-2204365.

\vspace{-2mm}

\section{Preliminaries}\label{sec:prelim}

We begin by recalling the necessary preliminary results regarding categories, weak factorization systems, and (pre)model structures. We refer the reader to~\cite[\S 2]{boor} for further details.

First recall that a category $\mathcal{C}$ consists of a class of \emph{objects} $\ob \mathcal C$ and for each pair $X,Y\in \ob \mathcal C$ a class of \emph{morphisms} $\mathcal C(X,Y)$ with \emph{source} $X$ and \emph{target} $Y$. When $f\in \mathcal C(X,Y)$, we say that $f\colon X\to Y$ is a morphism in $\mathcal C$. Furthermore, for all $X,Y,Z\in \ob C$, $\mathcal C$ is equipped with a \emph{composition} $\circ\colon \mathcal C(Y,Z)\times \mathcal C(X,Y)\to \mathcal C(X,Z)$ taking a pair $(g,f)$ of composable morphisms (target of $f$ equals source of $g$) to $g\circ f$. This composition must be associative, and for each $X\in \ob \mathcal C$ there is a distinguished identity morphism $\id_X\in \mathcal C(X,X)$ which is a two-sided identity for composition. We refer the reader to \cite{riehl2017category} for an accessible, contemporary introduction to category theory.

\begin{example}
The prototypical category is $\Set$, with $\ob \Set$ the proper class of sets and $\Set(X,Y)$ the set of functions with domain $X$ and codomain $Y$. Composition is usual composition of functions, and identity morphisms are the usual identity functions $\id_X\colon x\mapsto x$.
\end{example}

\begin{example}
Every poset (or, more generally, preorder) $(P,\leqslant)$ induces a category --- still called $P$ --- with objects the elements of $P$ and $P(x,y)$ a singleton set for $x\leqslant y$ and otherwise empty. Then $\id_x$ is the unique element of $P(x,x)$, and composition is uniquely determined by transitivity of $\leqslant$.
\end{example}

Model structures on categories are a convenient formalism for abstract homotopy theory. They are built from compatible pairs of weak factorization systems, which are defined in terms of lifting properties. We now introduce these concepts and provide relevant examples.

\begin{defn}
For any two morphisms $i \colon A \to B$ and $p \colon X \to Y$ in a category $\mathcal{C}$, we say that $i$ \emph{has the left lifting property (LLP) with respect to $p$}, or $p$ \emph{has the right lifting (RLP) property with respect to $i$},
if for all commutative squares of the form
\[
\begin{gathered}\xymatrix{
A \ar[r] \ar[d]_{i} & X \ar[d]^{p} \\ B \ar[r] & Y
}\end{gathered}    
\]
there exists a lift $h \colon B \to X$ which makes the resulting diagram commute. If $i$ lifts on the left of $p$ we write $i \boxslash p$. For any class $\mathcal{S}$ of morphisms in $\mathcal{C}$ we write
\begin{align*}
\mathcal{S}^{\boxslash}& = \{g \in \operatorname{Mor}(\mathcal{C}) \mid f \boxslash g \text{ for all } f \in \mathcal{S} \},\\
{}^{\boxslash} \mathcal{S} &= \{f \in \operatorname{Mor}(\mathcal{C}) \mid f \boxslash g \text{ for all } g \in \mathcal{S} \}.
\end{align*}
Note that $\mathcal{S} \subseteq {}^\boxslash \mathcal{T}$ if and only if $\mathcal{T} \subseteq \mathcal{S}^\boxslash$. We write $\mathcal{S}\boxslash \mathcal{T}$ when this holds.
\end{defn}

\begin{defn}
Let $\mathcal{L}$ and $\mathcal{R}$ be classes of morphisms in a category $\mathcal{C}$. Then the pair $(\mathcal{L},\mathcal{R})$ is a \emph{weak factorization system} if:
\begin{enumerate}
\item every morphism $f$ in $\mathcal{C}$ can be factored as $f=pi$ with $i \in \mathcal{L}$ and $p \in \mathcal{R}$,
\item $\mathcal{L} \boxslash \mathcal{R}$,
\item $\mathcal{L}$ and $\mathcal{R}$ are closed under retracts.
\end{enumerate}
\end{defn}

\begin{example}
Let $\Inj$ be the class of injective functions and $\Surj$ be the class of surjective functions. Then both pairs $(\Surj,\Inj)$ and $(\Inj,\Surj)$ form weak factorization systems on $\Set$.
\end{example}

\begin{defn}\label{def:premodel}
A \emph{premodel category} is a category $\mathcal{C}$ with all finite limits and colimits equipped with two weak factorization systems $(\LL,\RR)$ and $(\LL',\RR')$ such that $\RR \subseteq \RR'$ (equivalently $\LL \supseteq \LL'$). The \emph{weak equivalences} of a premodel structure is the class of maps $\mathcal{W} := \RR \circ \LL'$. For a fixed category $\mathcal{C}$ we denote by $\mathsf{P}(\mathcal{C})$ the collection of premodel structures of $\mathcal{C}$.
\end{defn}

Premodel structures were introduced in~\cite{barton} as a generalization of model structures (in the sense of Quillen) and are amenable to the theory of homotopy limits and colimits. The model structures appear as the premodel structures such that the weak equivalences $\mathcal{W}$ are closed under the 2-out-of-3 property. That is, if two of $f,g, f \circ g$ are weak equivalences, then so is the third. We shall write $\mathsf{Q}(\mathcal{C})$ for the collection of model structures on $\mathcal{C}$.

\begin{example}
In \cite{tobyomar}, Antol\'in Camarena and Barthel demonstrate that there are exactly nine model structures on $\Set$.
\end{example}

There is an inclusion $\mathsf{P}(\mathcal{C}) \subseteq \mathsf{Q}(\mathcal{C})$, which --- even for the most simple of $\mathcal{C}$ --- is a strict inclusion. In this paper, we wish to study a natural intermediate collection between these two which arises as a weakening of the two-out-of-three property.

\begin{defn}
Let $\mathcal{C}$ be a category equipped with a premodel structure. We say that the premodel structure is a \emph{composition closed premodel structure} if $\mathcal{W}$ is closed under composition. We write $\mathsf{C}(\mathcal{C})$ for the collection of composition closed premodel structures on $\mathcal{C}$.
\end{defn}

\begin{lemma}
Let $\mathcal{C}$ be a category. Then there are inclusions
\[
\mathsf{P}(\mathcal{C}) \subseteq \mathsf{C}(\mathcal{C}) \subseteq \mathsf{Q}(\mathcal{C}).
\]
\end{lemma}

Now that we have introduced the structures of interest in the general setting, let us restrict to the case where the category $\mathcal{C}$ is a finite lattice. In this case we have a concrete understanding of the collection of weak factorization systems via the language of \emph{transfer systems} as we now recall.

\begin{defn}\label{def:transfersystem}
Let $L = (\mathcal{P},\leqslant)$ be a lattice. A \emph{transfer system} on $L$ consists of a partial order $\mathcal{R}$ that refines $\leqslant$ and is closed under pullbacks: for all $x,y,z \in L$, if $x \,\mathcal{R}\, y$ and $z \leqslant y$, then $(x {\wedge} z) \, \mathcal{R} \, z$.
\end{defn}

In the following result, we consider poset structures on the collections of transfer systems and weak factorization systems. The order on transfer systems is by refinement. That is, we say $\mathcal{R} \leqslant \mathcal{R}'$ if and only if  for all $x,y \in L$, if $x \,\mathcal{R}\, y$ then $x \,\mathcal{R}'\, y$. We will denote the collection of transfer systems along with this ordering by $(\mathsf{Tr}(L), \leqslant)$. The order on weak factorization systems is given, as usual, by inclusion of the right class.

\begin{prop}[{\cite[Proposition 4.11, Theorem 4.13]{fooqw}}]\label{fooqwresult}
Let $L$ be a finite lattice and $\mathcal{R}$ a transfer system on $L$. Then there is a unique weak factorization system $(\mathcal{L},\mathcal{R})$ on $L$. In particular the assignment
\[
\mathcal{R} \longleftrightarrow ({}^\boxslash \mathcal{R}, \mathcal{R})
\]
is a poset isomorphism between the poset of transfer systems and the poset of weak factorization systems.
\end{prop}

\begin{remark}\label{rem:eitheror}
In light of~\autoref{fooqwresult} we will henceforth use the terms weak factorization system and transfer system interchangeably without any loss of generality. Unless the left class is explicitly required, we will simply write $\RR$ for the unique transfer system $(\LL, \RR)$ that it determines. When writing commutative diagrams, we will use the standard convention of writing $\xymatrix{x \ar@{->>}[r] & y}$ if $x \, \RR \, y$ and $\xymatrix{x \ar@{^(->}[r] &y}$ if $x \, \LL \, y$.
\end{remark}

The collection $\mathsf{Tr}(L)$ (and consequently the collection of weak factorization systems on $L$) has more structure than just being a poset. It is in fact a complete lattice.

\begin{lemma}[{\cite[Proposition 3.7]{fooqw}}]\label{lem:transislat}
Let $L$ be a finite lattice. Then the poset $(\mathsf{Tr}(L), \leqslant)$ is itself a finite lattice, and as such a complete lattice.
\end{lemma}

\begin{proof}
The greatest element in the finite poset $(\mathsf{Tr}(L), \leqslant)$ is the maximal transfer system $\mathcal{M}$ (i.e., $x \, \mathcal{M}\, y$ if and only if $x \leqslant y$). The meet operation is given by intersection. Thus by~\cite[Proposition 3.3.1]{stanley} the result follows, with the usual observation that any finite lattice is complete.
\end{proof}

From \autoref{lem:transislat}, and the description of premodel structures, we obtain the observation that premodel structures on a finite lattice can be realized by intervals in the lattice of transfer systems. For $(P, \leqslant)$ a poset, we denote by $\mathbb{I}(P)$ the poset of intervals of $L$. That is, $\mathbb{I}(P)$ is the poset whose elements are pairs $(a \leqslant_P b)$ with $a,b \in P$, and we say $(a \leqslant_P b) \leqslant_{\mathbb{I}(P)} (a' \leqslant_P b')$ if and only if $a \leqslant_P a'$ and $b \leqslant_P b'$.  

If $L$ is a complete lattice, then the collection of intervals is also a complete lattice. Indeed, for a collection of intervals $(a \leqslant_P b)_i$ we compute $\bigwedge (a \leqslant_P b)_i = (\bigwedge a_i) \leqslant_P (\bigwedge b_i)$ using the fact that $L$ is a lattice (and similarly for the join). As such, we obtain the result that for a fixed lattice $L$ the collection $\mathsf{P}(L)$ is itself a complete lattice.

\begin{lemma}\label{lem:premodelislat}
Let $L$ be a finite lattice. Then there is a bijection between the lattice of premodel structures and the lattice of intervals $\mathbb{I}(\mathsf{Tr}(L),\leqslant)$. In this lattice we have that for two premodel structures $P_1$ and $P_2$, $P_1 \leqslant P_2$ if and only if the identity functor $\mathrm{id} \colon P_1 \to P_2$ is right Quillen.
\end{lemma}

\begin{proof}
The statement regarding the  bijection between the lattice of premodel structures and the lattice of intervals $\mathbb{I}(\mathsf{Tr}(L),\leqslant)$ is immediate from construction. If $P_1 = (\RR_1 \leqslant \RR_1')$ and 
$P_2 = (\RR_2 \leqslant \RR_2')$, then unraveling the definition of the ordering on intervals, we have $P_1 \leqslant P_2$ if and only if $\RR_1 \leqslant \RR_2$ and $\RR_1' \leqslant \RR_2'$. Using terminology more familiar to (pre)model structure, this happens if and only if the identity functor preserves acyclic fibrations and fibrations. That is, the identity functor is right Quillen.
\end{proof}

\section{Composition closed premodel structures as lattice intervals}\label{sec:ccclosedaslattice}

In this section we will give conditions on a general premodel structure on a finite lattice such that the weak equivalences are closed under composition. We prove that there is an ordering which refines the natural inclusion ordering on weak factorization systems such that the intervals realize the composition closed premodel structures. 

Following \autoref{rem:eitheror} we can succinctly record the data of a premodel structure on a finite lattice as the pair $\RR \leqslant \RR'$ of weak factorization systems.

We begin with an observation regarding a characterization of when the weak equivalences --- which we recall are defined as $\RR \circ \LL'$ where $\LL' = {}^\boxslash \RR$ --- are closed under composition.

\begin{lemma}\label{lem:composition}
Let $\mathcal{C}$ be a category. Suppose that $\mathcal{A}, \mathcal{B} \subseteq \mathrm{Mor}(\mathcal{C})$ are two sets of morphisms containing all identity maps which are closed under composition. Then $\mathcal{B} \circ \mathcal{A}$ is closed under composition if and only if $\mathcal{A} \circ \mathcal{B} \subseteq \mathcal{B} \circ \mathcal{A}$.
\end{lemma}

\begin{proof}
By definition, $\mathcal{B} \circ \mathcal{A}$ is closed under composition if and only if $(\mathcal{B} \circ \mathcal{A}) \circ (\mathcal{B} \circ \mathcal{A}) = \mathcal{B} \circ \mathcal{A}$. Suppose that $\mathcal{B} \circ \mathcal{A}$ is closed under composition. Since both classes of morphisms contain all identity maps,
\[\mathcal{B} \circ \mathcal{A} \circ \mathcal{B} \circ \mathcal{A} \supseteq \mathcal{A} \circ \mathcal{B},\]
and it follows that $\mathcal{A} \circ \mathcal{B} \subseteq \mathcal{B} \circ \mathcal{A}$ as required.

Conversely, assume that $\mathcal{A} \circ \mathcal{B} \subseteq \mathcal{B} \circ \mathcal{A}$. Since $\mathcal{A}$ and $\mathcal{B}$ are both closed under composition we have that
\[\mathcal{B} \circ (\mathcal{A} \circ \mathcal{B}) \circ \mathcal{A} \subseteq \mathcal{B} \circ (\mathcal{B} \circ \mathcal{A}) \circ \mathcal{A} = \mathcal{B} \circ \mathcal{A}.\]
The result follows.
\end{proof}

We can now start moving towards our desired refined ordering on $\mathsf{Tr}(L)$ whose intervals detect the composition closed premodel structures. We highlight that the following results do not require any finiteness assumptions on the underlying lattice.

\begin{prop}\label{prop:transitiverel}
Let $\RR \leqslant \RR' \leqslant \RR''$ be weak factorization systems on an arbitrary lattice $L$ with corresponding left classes $\mathcal{L} \geqslant \mathcal{L}' \geqslant \mathcal{L}''$. Suppose further that $\RR \circ \mathcal{L}'$ and $\RR' \circ \mathcal{L}''$ are closed under composition. Then $\RR \circ \mathcal{L}''$ is closed under composition. In particular, being closed under composition is a transitive relation.
\end{prop}

\begin{proof}
By \autoref{lem:composition} we have that $\RR \circ \mathcal{L}''$ is closed under composition if and only if $\mathcal{L}'' \circ \RR \subseteq \RR \circ \mathcal{L}''$. Suppose $x \, \RR \, y$ and $y \, \mathcal{L}''\, z$. Since $\mathcal{L}'' \subseteq \mathcal{L}'$, we have $y \,\mathcal{L}' \,z$, and therefore by assumption there is some $w$ such that $x \,\mathcal{L}'\, w \,\RR \,z$. Similarly, since $\RR \subseteq \RR'$, there is some $w'$ such that $x \,\mathcal{L}'' \,w' \,\RR' \,z$. As such, we have a square
\[
\xymatrix{
x \ar[r]  \ar@{^(->}[d]_{\mathcal{L}'} & w' \ar@{->>}[d]^{\RR'} \\
w \ar@{..>}[ur] \ar[r] & z
}
\]
and we conclude that $w \leqslant w'$. However, $w \,\RR\, z$ and $w \leqslant w' \leqslant z$ implies that $w \,\RR \,w'$, and therefore $w \RR'' w'$. This gives us another square
\[
\xymatrix{
x \ar[r]  \ar@{^(->}[d]_{\mathcal{L}''} & w \ar@{->>}[d]^{\RR''} \\
w' \ar@{..>}[ur] \ar[r] & z \rlap{ .}
}
\]
It follows that $w=w'$, and thus $x \,\mathcal{L}'' \,(w' {=} w) \,\RR \,z$ as required.
\end{proof}

\begin{corollary}
There exists a partial order $\preccurlyeq$ on the set of weak factorization systems of an arbitrary lattice which refines $\leqslant$ such that $\RR \preccurlyeq \RR'$ if and only if $\RR \circ \mathcal{L}'$ is closed under composition.
\end{corollary} 

\begin{proof}
\autoref{prop:transitiverel} tells us that $\preccurlyeq$ is transitive. By the definition of a weak factorization system the relation is clearly reflexive and antisymmetric.
\end{proof}

Now that we have determined that the desired partial ordering $\preccurlyeq$ exists, we work towards a more concrete understanding of what it entails for $\RR$ and $\RR'$.

\begin{lemma}\label{lem:firstsplit}
Let $\RR \leqslant \RR'$. Then $\RR \preccurlyeq \RR'$ if and only if any square of the form
\begin{equation}\label{eq:diagram}
\begin{split}
\xymatrix{
x \ar[r]  \ar@{->>}[d]_{\RR} & z \ar@{->>}[d]^{\RR'} \\
y  \ar@{^(->}[r]_{\mathcal{L}'} & w
}
\end{split}
\end{equation}
has a splitting of the form
\[
\xymatrix{
x \ar[r]  \ar@{->>}[d]_{\RR} & z' \ar@{..>>}[dr]|{\RR} \ar[r] & z \ar@{->>}[d]^{\RR'} \\
y  \ar@{^(->}[rr]_{\mathcal{L}'} && w \rlap{ .}
}
\]
\end{lemma}

\begin{proof}
Suppose $\RR \preccurlyeq \RR'$ and that we have a square of the form (\ref{eq:diagram}). Since $x \, \RR\, y  \, \mathcal{L}' \, w$ and $\RR \preccurlyeq \RR'$ by assumption, we also have $x \,\mathcal{L}' \,z' \,\RR \,w$ for some $z'$. Thus we have a square
\[
\xymatrix{
x \ar@{^(->}[d]_{\mathcal{L}'}\ar[r] & z \ar@{->>}[d]^{\RR'} \\
z' \ar[r] \ar@{..>}[ur] & w
}
\]
which implies that $x \leqslant z' \leqslant z$. Since $z' \,\RR\, w$, one obtains the desired splitting.

Conversely, suppose $\RR \leqslant \RR'$ satisfies the stated condition and suppose $x \,\RR \,y\, \mathcal{L}' \,w $. Then we can factor $x \to w$ as $x \,\mathcal{L}' \,z\, \RR' \,w$, giving a square
\[
\xymatrix{
x \ar@{^(->}[r]^{\mathcal{L}'} \ar@{->>}[d]_{\RR}& z \ar@{->>}[d]^{\RR'}\\
y \ar@{^(->}[r]_{\mathcal{L}'}& w\rlap{ .}
}
\]
By assumption we can find a splitting
\[
\xymatrix{
x \ar[r] \ar@{->>}[d]_{\RR}& z'\ar[r]  \ar@{..>>}[dr]|{\RR} & z \ar@{->>}[d]^{\RR'}\\
y \ar@{^(->}[rr]_{\mathcal{L}'}&& w
\rlap{ .} }
\]
As $x \,\mathcal{L}' \,z$ and $x \leqslant z' \leqslant z$, closure of the left class under pushouts gives us $z' \,\mathcal{L'} \,z$. However we also have $z' \,\RR \,w$ and $z' \leqslant z \leqslant w$, and  so by closure of the right class under pullbacks we have $z' \,\RR \,z$. As such it follows that $z' = z$, and hence $x \,\mathcal{L}' \,z\, \RR \,w$.
\end{proof}

\begin{remark}
From the proof, we see that it is also the case that the relation $x\to z'$ in the splitting is always in $\mathcal{L}'$.
\end{remark}

We now give a second condition on when $\RR \preccurlyeq \RR'$ which we will see is more amenable to computation.

\begin{prop}\label{prop:onfinite}
Let $\RR \leqslant \RR'$. Then $\RR \preccurlyeq \RR'$ if and only if any square
\[
\xymatrix{
x \ar[r] \ar@{->>}[d]_{\RR} & z \ar@{->>}[d]^{\RR'} \\
y \ar[r] & w
}
\]
has a splitting of the form
\[
\xymatrix{
x \ar[r] \ar@{->>}[d]_{\RR} & z' \ar[r] \ar@{->>}[d]^{\RR} & z \ar@{->>}[d]^{\RR'} \\
y \ar[r] & w'\ar@{->>}[r]_{\RR'} & w
}
\]
\end{prop}

\begin{proof}
Suppose $\RR \leqslant \RR'$ satisfies the above condition. Then for any square
\[
\xymatrix{
x \ar[r] \ar@{->>}[d]_{\RR} & z \ar@{->>}[d]^{\RR'} \\
y \ar@{^(->}[r]_{\mathcal{L}'} & w \rlap{ ,}
}
\]
we obtain a splitting
\[
\xymatrix{
	x \ar[r] \ar@{->>}[d]_{\RR} & z' \ar[r] \ar@{->>}[d]^{\RR} & z \ar@{->>}[d]^{\RR'} \\
	y \ar[r] & w'\ar@{->>}[r]_{\RR'} & w \rlap{ .}
}
\]
But then we have a diagram
\[
\xymatrix{
	y \ar[r] \ar@{^(->}[d]_{\LL'} & w' \ar@{->>}[d]^{\RR} \\
	w \ar[r]\ar@{..>}[ru] & w
}
\]
and hence $w' = w$, so we've already obtained the desired splitting required by \autoref{lem:composition}, and as such $\RR \preccurlyeq \RR'$.

Conversely, assume that $\RR \preccurlyeq \RR'$, and suppose we have a square
\[
\xymatrix{
x \ar[r] \ar@{->>}[d]_{\RR} & z \ar@{->>}[d]^{\RR'} \\
y \ar[r] & w \rlap{ .}
}
\]
We can factor $y \to w$ as $y \, \mathcal{L}' \,w' \,\RR' \,w$, and we have $\widetilde{z} := (z {\wedge} w') \, \RR' \, w'$ by closure under pullbacks. We obtain a square
\[
\xymatrix{
x \ar[r] \ar@{->>}[d]_{\RR} & \widetilde{z} \ar@{->>}[d]^{\RR'} \\
y \ar@{^(->}[r]_{\mathcal{L}'} & w' \rlap{ ,}
}
\]
and hence by \autoref{lem:firstsplit} we can find a splitting
\[
\xymatrix{
x \ar[r] \ar@{->>}[d]_{\RR}& z'\ar[r]  \ar@{..>>}[dr]|{\RR} & \widetilde{z} \ar@{->>}[d]^{\RR'}\\
y \ar@{^(->}[rr]_{\mathcal{L}'}&& w'}
\]
which gives
\[
\xymatrix{
x \ar[r] \ar@{->>}[d]_{\RR} & z' \ar[r] \ar@{->>}[d]^{\RR} & z \ar@{->>}[d]^{\RR'} \\
y \ar[r] & w' \ar@{->>}[r]_{\RR'} & w
}
\]
as required.
\end{proof}

The statement of \autoref{prop:onfinite} is particularly useful as it doesn't make any reference to cofibrations. In fact the structure of the constraint in \autoref{prop:onfinite} enables us to prove for any finite lattice $L$, that $(\mathsf{Tr}(L), \preccurlyeq)$, is not just a poset but a lattice.

\begin{corollary}\label{cor:closure}
Let $L$ be a finite lattice and $\RR \preccurlyeq \RR',\RR''$. Then $\RR\preccurlyeq\RR'\cap\RR''$.
\end{corollary}

\begin{proof}
Suppose we have a square
\[
\xymatrix{
	x \ar[r] \ar@{->>}[d]_{\RR} & z \ar@{->>}[d]^{\RR'\cap\RR''} \\
	y \ar[r] & w \rlap{ ,}
}
\]
and let $\tilde{z} \,\RR \, \tilde{w}$ be maximal among splittings
\[
\xymatrix{
	x \ar[r] \ar@{->>}[d]_{\RR} & \tilde{z} \ar[r] \ar@{->>}[d]^{\RR} & z \ar@{->>}[d]^{\RR'\cap\RR''} \\
	y \ar[r] & \tilde{w}\ar[r] & w \rlap{ .}
}
\]

By \autoref{prop:onfinite}, we have two splittings
\[
\xymatrix{
	\tilde{z} \ar[r] \ar@{->>}[d]_{\RR} & z' \ar[r] \ar@{->>}[d]^{\RR} & z \ar@{->>}[d]^{\RR'} \\
	\tilde{w} \ar[r] & w'\ar@{->>}[r]_{\RR'} & w
}
\]
\[
\xymatrix{
	\tilde{z} \ar[r] \ar@{->>}[d]_{\RR} & z'' \ar[r] \ar@{->>}[d]^{\RR} & z \ar@{->>}[d]^{\RR''} \\
	\tilde{w} \ar[r] & w''\ar@{->>}[r]_{\RR''} & w \rlap{ .}
}
\]
By maximality of $\tilde{z}\to\tilde{w}$ we must have $\tilde{z} = z' = z''$ and $\tilde{w} = w' = w''$, and hence $\tilde{w}\to w\in\RR'\cap\RR''$.
\end{proof}

This observation brings us to the main result of this section.

\begin{theorem}\label{thm:islattice}
Let $L$ be a finite lattice. Then $(\mathsf{Tr}(L), \preccurlyeq)$ is a finite lattice.  As such, the collection $\mathsf{C}(L)$ of composition closed premodel structures on $L$ is a finite lattice.
\end{theorem}

\begin{proof}
The poset $(\mathsf{Tr}(L), \preccurlyeq)$ has joins by \autoref{cor:closure}, since we can take $\RR\vee\RR'$ to be the intersection of all transfer systems $\RR''$ with $\RR\preccurlyeq\RR''$ and $\RR'\preccurlyeq\RR''$. Since it also has a minimal element given by the trivial transfer system, $(\mathsf{Tr}(L), \preccurlyeq)$ forms a lattice. The claim regarding $\mathsf{C}(L)$ then follows again from the observation that the collection of intervals on a finite lattice is once again a finite lattice.
\end{proof}

\begin{remark}
The finiteness assumption in \autoref{thm:islattice} is essential as we now demonstrate.  Let $L$ be the lattice $(\mathbb{N} \cup \{\infty \}) \times [1]$ equipped with the following transfer systems $\RR_1 \leqslant \RR_2$.
\[
\xymatrix{
x_1 \ar[r] \ar@{->>}[d]^{\RR_1} & \ar[r] x_2 \ar[r] \ar@{->>}[d]^{\RR_1} & x_3 \ar[r] \ar@{->>}[d]^{\RR_1} & \cdots \ar[r] & x_\infty \ar[d] \\
y_1 \ar[r] & \ar[r] y_2 \ar[r] & y_3 \ar[r] & \cdots \ar[r] & y_\infty
}
\]
\[
\xymatrix{
x_1 \ar[r] \ar@{->>}[d]^{\RR_2} & \ar[r] x_2 \ar[r] \ar@{->>}[d]^{\RR_2} & x_3 \ar[r] \ar@{->>}[d]^{\RR_2} & \cdots \ar[r] & x_\infty \ar@{->>}[d]^{\RR_2} \\
y_1 \ar[r] & \ar[r] y_2 \ar[r] & y_3 \ar[r] & \cdots \ar[r] & y_\infty
}
\]
Suppose that $\RR_1 \preccurlyeq \RR'$ and $\RR_2 \leqslant \RR'$ for some $\RR'$. Suppose that there exists some $n$ such that for all $m \geqslant n$, $y_m \, \cancel{\RR'} \, y_\infty$. Then we have $y_n \,\mathcal{L}' \,y_\infty$ and $x_n \, \RR_1 \, y_n$, and since $\RR_1 \preccurlyeq \RR'$, we must have $x_n \, \mathcal{L}' \,z\, \RR_1 \,y_\infty$. However, the only $z$ with $z \,\RR_1 \,y_\infty$ is $z=y_\infty$. On the other hand, since $\RR_2 \leqslant \RR'$, we have a square
\[
\xymatrix{
x_n \ar[r] \ar[d]& x_\infty \ar@{->>}[d]^{\RR'} \\
y_\infty \ar[r] & y_\infty
}
\]
which clearly admits no lift, and hence $x_n \, \cancel{\mathcal{L}'} \, y_\infty$, which is a contradiction.

Conversely, let $(n) = (n_1, n_2, \dots)$ be any increasing infinite sequence, and define $\RR_{(n)}$ to be the smallest transfer system $\RR'$ containing $\RR_2$ such that $y_{n_k} \, \RR' \,y_\infty$ for all $k$. Specifically, we have $x_m \,\RR_{(n)}\, y_m$ for all $m$ (including $m=\infty$), and for all $k$ and all $m \geqslant n_k$ (including $m = \infty$) we have $x_{n_k} \,\RR_{(n)} \,x_m$, $x_{n_k} \,\RR_{(n)} \,y_m$ and $y_{n_k} \,\RR_{(n)} \,y_m$. These are the only relations in $\RR_{(n)}$.

We claim that $\RR_1 \preccurlyeq \RR_{(n)}$ and $\RR_2 \preccurlyeq \RR_{(n)}$ for any such $\RR_{(n)}$ constructed this way. Indeed, for any such $\RR_{(n)}$, one can explicitly determine $\mathcal{L}_{(n)}$: namely we have $x_n \, \mathcal{L}_{(n)} \, x_m$ and $y_n \, \mathcal{L}_{(n)}\, y_m$ if and only if $n_k < n \leqslant m \leqslant y_{k+1}$ for some $k$, and no other arrows exist in $\mathcal{L}_{(n)}$. With this explicit description, one is able to verify that $\RR_1 \circ \mathcal{L}_{(n)}$ and $\RR_2 \circ \mathcal{L}_{(n)}$ are both closed under composition. 

Now we know that if $\RR_1 \preccurlyeq \RR'$ and $\RR_2 \preccurlyeq \RR'$ then $\RR'$ must contain some $\RR_{(n)}$, so in particular if the join $\RR_1 \vee \RR_2$ exists then it must be of the form $\RR_{(n)}$ for some sequence $(n) = (n_1, n_2, \dots)$ by minimality. But then if we define $(n')$ to be the same sequence with $n_1$ removed, we have $\RR_{(n')} < \RR_{(n)}$ and $\RR_1, \RR_2 \preccurlyeq \RR_{(n')}$, contradicting minimality. Hence, the element $\RR_1 \vee \RR_2$ does not exist.
\end{remark}

\begin{remark}
Recall that in \autoref{lem:premodelislat} we observed that the collection of premodel structures and right Quillen functors between them formed a complete lattice. In particular, any limit or colimit of premodel structures on a fixed finite lattice $L$ along right Quillen functors exists.

We have proved in~\autoref{thm:islattice} that the collection of composition closed model structures also forms a complete lattice. However, it does not follow that any limit or colimit of composition closed premodel structures on a fixed finite lattice $L$ along right Quillen functors exists. This is due to the fact that for two composition closed premodel structures $C_1 = (\RR_1 \preccurlyeq \RR_1')$  and $C_2 = (\RR_2 \preccurlyeq \RR_2')$  we have $C_1 \leqslant_{\mathbb{I}(\mathsf{Tr}(L),\preccurlyeq)} C_2$ if and only if $\RR_1 \preccurlyeq \RR_2$ and $\RR_1' \preccurlyeq \RR_2'$. In particular, this is the data of a right Quillen functor satisfying an additional condition. Let us locally call such a functor a \emph{restricted right Quillen functor}.

We have not been able to find a known description of what it means for a right Quillen functor --- even between model categories --- to be restricted right Quillen. Take for example the following model structures on the lattice $[2]$ and the right Quillen functor between them:
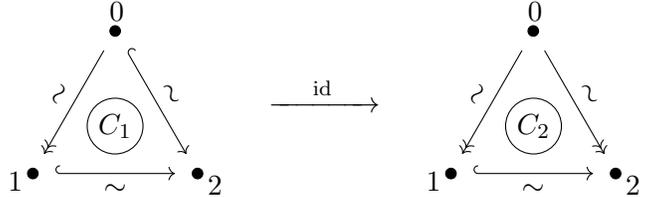
\[
\begin{gathered}
\begin{tikzpicture}%%model5
\node (pol) [draw=none, minimum size=2.5cm, regular polygon, regular polygon sides=3] at (0,0) {};
\foreach \n [count=\nu from 0, remember=\n as \lastn, evaluate={\nu+\lastn}] in {1,2,...,3}
\node[anchor=\n*(360/3)-215] at (pol.corner \n){$\nu$};
\foreach \n in {1,2,...,3}
\draw[fill = black] (pol.corner \n) circle (2pt);

\draw[->>, shorten >=3mm, shorten <=3mm](pol.corner 1) -- (pol.corner 2) node[midway,sloped,above]{$\sim$};

\draw[right hook->, shorten >=3mm, shorten <=3mm](pol.corner 1) -- (pol.corner 3) node[midway,sloped,above]{$\sim$};

\draw[right hook->, shorten >=3mm, shorten <=3mm](pol.corner 2) -- (pol.corner 3) node[midway,sloped,below]{$\sim$};
\node at (0,-0.0) {\circled{$C_1$}};

\end{tikzpicture}
\end{gathered}
\quad
\xrightarrow{\hspace{0.5cm}\mathrm{id}\hspace{0.5cm}}
\quad
\begin{gathered}
\begin{tikzpicture}%%Model7
\node (pol) [draw=none, minimum size=2.5cm, regular polygon, regular polygon sides=3] at (0,0) {};
\foreach \n [count=\nu from 0, remember=\n as \lastn, evaluate={\nu+\lastn}] in {1,2,...,3}
\node[anchor=\n*(360/3)-215] at (pol.corner \n){$\nu$};
\foreach \n in {1,2,...,3}
\draw[fill = black] (pol.corner \n) circle (2pt);

\draw[->>, shorten >=3mm, shorten <=3mm](pol.corner 1) -- (pol.corner 2) node[midway,sloped,above]{$\sim$};

\draw[->>, shorten >=3mm, shorten <=3mm](pol.corner 1) -- (pol.corner 3) node[midway,sloped,above]{$\sim$};

\draw[right hook->, shorten >=3mm, shorten <=3mm](pol.corner 2) -- (pol.corner 3) node[midway,sloped,below]{$\sim$};
\node at (0,-0.0) {\circled{$C_2$}};
\node at (1.6,-0.9) {.};
\end{tikzpicture}
\end{gathered}
\]
Then this functor is not restricted right Quillen, even though it has the strong property of preserving weak equivalences. We can see the failure of this to be restricted right Quillen using \autoref{lem:firstsplit}. In particular, Setting $x=z=0$, $y=1$, and $w=2$, we have the square
\[
\xymatrix{
0 \ar@{=}[r]  \ar@{->>}[d]_{\RR} & 0 \ar@{->>}[d]^{\RR'} \\
1  \ar@{^(->}[r]_{\mathcal{L}'} & 2
}
\]
which does not admit a splitting
\[
\xymatrix{
0 \ar@{=}[r]  \ar@{->>}[d]_{\RR} & 0 \ar@{..>}[dr]|{\RR} \ar@{=}[r] & 0 \ar@{->>}[d]^{\RR'} \\
1  \ar@{^(->}[rr]_{\mathcal{L}'} && 2
}
\]
as $0$ is not fibrant in $C_1$.
\end{remark}

\subsection{A refined ordering for model structures}\label{sec:refinedformodel}

So far we have seen that premodel structures  (resp., composition closed premodel structures) on a finite lattice $L$ are in bijection with intervals of the lattice $(\mathsf{Tr}(L),\leqslant)$ (resp., intervals of the lattice $(\mathsf{Tr}(L),\preccurlyeq)$). One may hope for a similar result to detect the model structures. Here we will show that such a refined ordering on $\mathsf{Tr}(L)$ does exist, but the resulting poset is not a lattice. The proof proceeds as in \autoref{prop:transitiverel} where we show that the desired relation is transitive.

\begin{prop}
Let $\RR \leqslant \RR' \leqslant \RR''$ be transfer systems on an arbitrary lattice such that $(\RR, \RR')$ and $(\RR', \RR'')$ are model structures. Then $(\RR, \RR'')$ is a model structure.
\end{prop}

\begin{proof}
We have by \autoref{prop:transitiverel} that $\RR \preccurlyeq \RR''$, so it suffices to show that if $f \circ g \in \RR \circ \LL''$ then $f \in \RR \circ \LL''$ if and only if $g \in \RR \circ \LL''$. By duality it suffices to show that $g \in \RR \circ \LL''$ implies that $f \in \RR \circ \LL''$. To this end, suppose that $ x \to y \in \RR \circ \LL''$, and $x \to z \in \RR \circ \LL''$, with $z \leqslant y$, we wish to show that $z \to y \in \RR \circ \LL''$.

Since $\LL'' \subseteq \LL'$, we know that $x \to y$ and $x \to z$ are in $\LL' \circ \RR$, so since $(\RR, \RR')$ forms a model structure, we have $z \, \LL' \, u \, \RR \,y$ for some $u$. Similarly using that $\RR \subseteq \RR'$ and $(\RR', \RR'')$ forms a model structure, we can find some $v$ such that $z \,\LL'' \,v\, \RR' \,y$.

By pullback closure, we have $(u {\wedge} v) \,\RR' \,u$, and since $z \,\LL'\, u$ and $z \leqslant u {\wedge} v$,  we have a lift $u \leqslant u {\wedge} v$, which implies that $u \leqslant v$. Similarly we have $(u {\wedge} v) \,\RR\, v$ and $z \,\LL'' \,v$, and since $\LL'' \subseteq \LL$, the same argument forces $v \leqslant u$. Thus $z \,\LL'' \,(v{=}u) \,\RR \,y$ as required. 
\end{proof}

\begin{corollary}
There exists a partial ordering $\sqsubseteq$ on the set of transfer systems of an arbitrary lattice which refines $\preccurlyeq$ such that $\RR \sqsubseteq \RR'$ if and only if the pair forms a model structure.
\end{corollary}

\begin{remark}
Although we now have a partial ordering $\sqsubseteq$ whose intervals detect model structures, the resulting object is not a lattice. This should not be surprising as we know that the collection of model structures even on a fixed underlying category does not admit all homotopy limits and colimits.

Let us provide an example to show that $(\mathsf{Tr}(L), \sqsubseteq)$ fails to be a lattice even in a relatively trivial case. Set $L=[2]$. Then $(\mathsf{Tr}([2]), \sqsubseteq)$ takes on the form
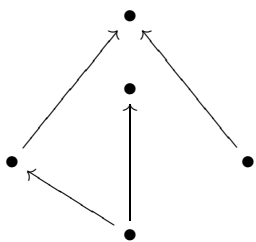
\begin{figure}[H]
    \centering \[
    \xymatrix@C=3.0em@R=1.5em{
    &\bullet \\
    &\bullet \\
    \bullet \ar[uur] & & \bullet \ar[uul]\\
    & \bullet \ar[uu] \ar[ul]
    }
    \]
\end{figure}
\noindent which is clearly not a lattice.
\end{remark}

\section{The case of finite total orders}\label{sec:caseoffinitetotal}

We will now restrict the theory from \autoref{sec:ccclosedaslattice} to the case where $L = [n]$, a finite total order. In~\cite{boor}, the premodel structures on $[n]$ were proved to be in bijection with the intervals of the Tamari lattice. Here we will show that the refined ordering $\preccurlyeq$ retrieves  a well-known refinement of the Tamari lattice, namely that of the \emph{Kreweras ordering}, which demonstrates the rich structures that one can possibly obtain when considering composition closed premodel structures.

\subsection{Noncrossing partitions and the Kreweras ordering}

\begin{defn}
A partition of the set $\{0,1, \dots n\}$ is said to be \emph{noncrossing} if for all $0 \leqslant a < b < c < d \leqslant n$ such that $a$ and $c$ are in the same block and $b$ and $d$ are in the same block, then $a,b,c$ and $d$ are all in the same block.
\end{defn}

There is a natural ordering on the collection $\mathsf{NC}_{n+1}$ of noncrossing partitions on the set $\{0,1 , \dots , n\}$ given by refinement of partitions. This ordering is called the \emph{Kreweras ordering} which results in a lattice (the \emph{Kreweras lattice})~\cite{kreweras}. The Tamari lattice is a strict extension of the Kreweras lattice (i.e., the Kreweras lattice is a refinement of the Tamari lattice).

\begin{example}
Let us consider the case of $[2]$. Then the Tamari and Kreweras lattices take the form
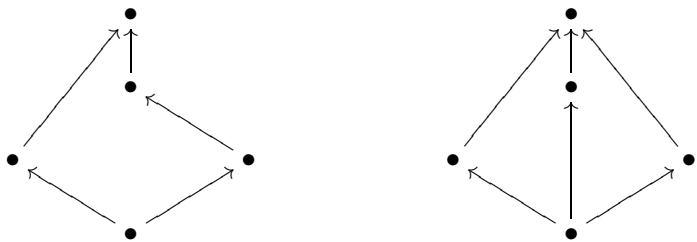
\begin{figure}[H]
    \centering \[
    \xymatrix@C=3.0em@R=1.5em{
    &\bullet \\
    &\bullet \ar[u] \\
    \bullet \ar[uur] & & \bullet \ar[ul]\\
    & \bullet \ar[ur] \ar[ul]
    }
\qquad \qquad \qquad
    \xymatrix@C=3.0em@R=1.5em{
    &\bullet \\
    &\bullet \ar[u] \\
    \bullet \ar[uur] & & \bullet \ar[uul]\\
    & \bullet \ar[ur] \ar[ul] \ar[uu]
    }
    \]
\end{figure}
\noindent respectively. We note that there are 13 intervals in the Tamari lattice, while there are only 12 in the Kreweras lattice.
\end{example}

We recall the bijection between noncrossing partitions and transfer systems from~\cite{fooqw}.

\begin{defn}
Let $\RR$ be a transfer system on $[n]$. We can define a non-decreasing function $\pi_\RR \colon [n] \to [n]$ by setting $\pi_\RR(i)$ to be the maximal $j$ such that $i \,\RR\, j$.
\end{defn}

\begin{prop}[{\cite[Theorem 5.7]{fooqw}}]\label{prop:ncistr}
Let $\RR$ be a transfer system on $[n]$, then the nonempty fibers of $\pi_\RR$ form a noncrossing partition of $[n]$. Moreover, this provides a bijection between sets of noncrossing partitions and transfer systems on $[n]$.
\end{prop}

In~\cite{boor} it was proved that the lattice $(\mathsf{Tr}([n]), \leqslant)$ is isomorphic to the Tamari lattice. We can instead use the Kreweras ordering, pulling back along the bijection above. In particular we have the following.

\begin{lemma}
Let $\RR$ and $\RR'$ be transfer systems on $[n]$. Then $\RR \leqslant \RR'$ in the Kreweras lattice if and only if $\pi_\RR(i) = \pi_\RR(j)$ implies $\pi_{\RR'}(i) = \pi_{\RR'}(j)$ for all $i,j \in [n]$.
\end{lemma}

For $\RR$ and $\RR'$ as in the above lemma, we will, for now, write $\RR \leqslant_K \RR'$ for the order relation in the Kreweras ordering. The main result of this section is to prove that $\leqslant_K \,\equiv\, \preccurlyeq$, that is, composition closed model structures on $[n]$ are in bijection with intervals in the Kreweras lattice.

\begin{theorem}\label{thm:main}
Let $\RR \leqslant \RR'$ be transfer systems on $[n]$. Then $\RR \preccurlyeq \RR'$ if and only if $\pi_{\RR} \leqslant \pi_{\RR'}$. That is, the ordering $\preccurlyeq$ is exactly the Kreweras ordering.
\end{theorem}

\begin{proof}
Observe first that $\pi_\RR\leqslant\pi_{\RR'}$ if and only if $\pi_{\RR'}(\pi_{\RR}(i)) = \pi_{\RR'}(i)$ for all $i$. Indeed, if $\pi_\RR\leqslant\pi_{\RR'}$ then by observing that we always have $\pi_\RR(i) = \pi_{\RR}(\pi_{\RR}(i))$ by transitivity of $\RR$, we obtain $\pi_{\RR'}(\pi_{\RR}(i)) = \pi_{\RR'}(i)$. Conversely it suffices to observe that $\pi_\RR(i) = \pi_\RR(j)$ implies
\[\pi_{\RR'}(i) = \pi_{\RR'}(\pi_{\RR}(i)) = \pi_{\RR'}(\pi_{\RR}(j)) = \pi_{\RR'}(j).\]

Now if $\RR\leqslant\RR'$ then for all $i$ we have a diagram
\[
\xymatrix{
	i \ar[r] \ar@{->>}[d]_{\RR}& i \ar@{->>}[d]^{\RR'} \\
	\pi_{\RR}(i) \ar[r] & \pi_{\RR'}(i).
}
\]
If $\RR\preccurlyeq\RR'$ then by \autoref{prop:onfinite} we obtain a splitting
\[
\xymatrix{
	i \ar[r] \ar@{->>}[d]_{\RR}& i\ar[r] \ar@{->>}[d]^{\RR} & i \ar@{->>}[d]^{\RR'} \\
	\pi_{\RR}(i) \ar[r] & j\ar@{->>}[r]_-{\RR'}& \pi_{\RR'}(i)
}
\]
but by maximality of $\pi_\RR(i)$ we must have $j = \pi_\RR(i)$ and hence $\pi_{\RR'}(\pi_\RR(i))\geqslant\pi_{\RR'}(i)$. But $i\, \RR\,\pi_{\RR}(i) \,\RR' \,\pi_{\RR'}(\pi_\RR(i))$, so the reverse inequality follows immediately.

Conversely assume $\pi_\RR\leqslant\pi_{\RR'}$ and suppose we're given a square
\[
\xymatrix{
	i \ar[r] \ar@{->>}[d]_{\RR}& j \ar@{->>}[d]^{\RR'} \\
	k \ar[r] & \ell.
}
\]
If $j\geqslant k$ then we obtain a trivial splitting
\[
\xymatrix{
	i \ar[r] \ar@{->>}[d]_{\RR}& j\ar[r] \ar@{->>}[d]_{\RR} & j \ar@{->>}[d]^{\RR'} \\
	k \ar[r] & j\ar@{->>}[r]_{\RR'} & \ell.
}
\]
Otherwise since $[n]$ is totally-ordered we have $i\leqslant j<k$ and hence by pullback-closure we have $i \, \RR \, j \,\RR' \,\ell$ and hence $\pi_{\RR'}(\pi_\RR(i)) = \pi_{\RR'}(i) \geqslant\ell$. If $\pi_\RR(i)\geqslant\ell$ then we obtain another trivial splitting
\[
\xymatrix{
	i \ar[r] \ar@{->>}[d]_{\RR}& i\ar[r] \ar@{->>}[d]_{\RR} & j \ar@{->>}[d]^{\RR'} \\
	k \ar[r] & \ell\ar@{->>}[r]_{\RR'} & \ell.
}
\]
Else, we obtain a non-trivial splitting
\[
\xymatrix{
	i \ar[r] \ar@{->>}[d]_{\RR}& i\ar[r] \ar@{->>}[d]_{\RR} & j \ar@{->>}[d]^{\RR'} \\
	k \ar[r] & \pi_\RR(i)\ar@{->>}[r]_{\RR'} & \ell
}
\]
and the result follows.
\end{proof}

\begin{corollary}\label{cor:thecount}
Let $n \geqslant 0$. Then
\[
|\mathsf{C}([n])| = \dfrac{1}{2n+3} \binom{3n+3}{n+1}.
\]
\end{corollary}

\begin{proof}
By \autoref{thm:main} we have $|\mathsf{C}([n])|$ is the number of intervals in the Kreweras lattice of noncrossing partitions on $[n]$. The result then follows from the computation of intervals from~\cite{kreweras}.
\end{proof}

\begin{remark}\label{rem:compatible}
We return to the remark in the introduction regarding compatible pairs of $N_\infty$ operads as introduced in~\cite{MR4327103}. In essence, a compatible pair of $N_\infty$ operads is a pair $(\mathcal{O}_1, \mathcal{O}_2)$ of $N_\infty$ operads satisfying  $\mathcal{O}_1 \leqslant \mathcal{O}_2$  in addition to a condition regarding coinduction of groups. In~\cite{chan} this condition was passed through the equivalence to transfer systems to give a more checkable condition which we now recall (in a purely categorical fashion).  Let $L$ be a finite lattice. A pair of transfer systems $(\RR,\RR')$ on $L$ are \emph{compatible} if:
\begin{enumerate}
    \item $\RR \leqslant \RR'$;
    \item For all $x \in L$ if $y,z \leqslant x$ are such that $y \, \RR \, x$ and $(y {\wedge} z) \, \RR' \, y$, then $z \RR' x$.
\end{enumerate}
Let us write $\mathsf{Com}(L)$ for the collection of compatible pairs for a fixed lattice $L$. One immediately sees that the pair $(\RR , \RR)$ need not be compatible, and, in particular, the collection of premodel structures and compatible pairs on a lattice cannot be the same. However, in the case of $[n]$, Hill, Meng, and Li \cite{compatible} have computed
\[
|\mathsf{Com}([n])| = \dfrac{1}{2n+3} \binom{3n+3}{n+1}.
\]
That is, $|\mathsf{Com}([n])| = |\mathsf{C}([n])|$. This suggests the existence of an explicit bijection between composition closed premodel structures on $[n]$ and the compatible pairs of transfer systems. We warn, however, that for lattices other than $[n]$, computational evidence suggests that there are always more compatible pairs than composition closed premodel structures.
\end{remark}

\subsection{Enumeration and asymptotics}

Using \autoref{thm:main} along with the results of~\cite{boor} we can provide numerical results regarding the density of composition closed premodel structures among premodel structures and model structures among composition closed premodel structures. 

We remind the reader that we write $\mathsf{P}([n])$ (resp., $\mathsf{C}([n])$, $\mathsf{Q}([n])$) for the collection of premodel structures (resp., composition closed premodel structures, model structures) on $[n]$.

We recall from~\cite{boor} that
\[
|\mathsf{P}([n])| = \dfrac{2}{(n+1)(n+2)}\binom{4n+5}{n}
\]
which follows from the count of intervals in the Tamari lattice achieved by~\cite{chapoton}, and
\[
|\mathsf{Q}([n])| = \binom{2n+1}{n}.
\]
In~\cite{boor} it is further shown that
\[
\dfrac{|\mathsf{Q}([n])|}{|\mathsf{P}([n])|} \sim c 2^{dn} n ^2.\]
where
\begin{align*}
c=& \dfrac{243 \sqrt{3/2}}{1024} \approx 0.290638, \\
d=& 3 \log_2(3) - 6 \approx -1.24511.
\end{align*}

From \autoref{cor:thecount} we have
\[
|\mathsf{C}([n])| = \dfrac{1}{2n+3} \binom{3n+3}{n+1}.
\]
By Stirling's approximation we have
\[
\binom{n}{k} \sim \sqrt{\dfrac{n}{2 \pi k (n-k)}} \cdot \dfrac{n^n}{k^k(n-k)^{n-k}}
\]
for both $k$ and $n$ large. Applying this approximation to $\mathsf{Q}([n])$ and $\mathsf{C}([n])$ yields
\begin{align*}
\dfrac{|\mathsf{Q}([n])|}{|\mathsf{C}([n])|} \sim & \dfrac{4 \sqrt{\pi + \pi n}\left(4 \sqrt{2}n^2 + 8\sqrt{2}n + 3 \sqrt{2}\right)(2n+2)^{2n} (2n+1)^{2n} \sqrt{\dfrac{2n+1}{\pi n^2 + \pi n}}}{27\left(\sqrt{3}n + \sqrt{3}\right)(3n+3)^{3n} n^n} \\[10pt]
\sim& \dfrac{32 \sqrt{3} \sqrt{\pi n} (2n)^{4n} n \sqrt{\dfrac{1}{\pi n}}}{81 (3n)^{3n} n^n} \\[10pt]
=& \dfrac{32 \sqrt{3}}{81} \left( \dfrac{16}{27} \right)^n n.
\end{align*}
In other words we have
\[
\dfrac{|\mathsf{Q}([n])|}{|\mathsf{C}([n])|} \sim c' 2^{d' n} n
\]
where
\begin{align*}
c' = & \dfrac{32 \sqrt{3}}{81} \approx 0.6842670, \\
d' = & 4 - 3 \log_2(3) \approx - 0.754888.
\end{align*}
As such, we see that $\frac{|\mathsf{Q}([n])|}{|\mathsf{C}([n])|}$ is asymptotically exponential decay times a linear term, and therefore approaches 0 swiftly for large $n$.

We can now solve for
\[
\dfrac{|\mathsf{C}([n])|}{|\mathsf{P}([n])|} \sim c'' 2^{d'' n} n
\]
where
\begin{align*}
c'' = & \dfrac{c}{c'} = \dfrac{19683}{65536}\sqrt{2} \approx 0.424743, \\
d'' = & d-d' =  6 \log_2(3) - 10 \approx - 0.490255.
\end{align*}
In conclusion, we see that composition closed premodel structures are rare among premodel structures, and moreover model structures are rare among composition closed premodel structures.

\section{Application: Model structures on $[n]$ as tricolored trees}\label{sec:stacks}

In the previous section we produced a bijection between composition closed premodel structures on $[n]$ and Kreweras intervals. In this section we will relate composition closed premodel structures to tricolored trees using the work~\cite{intervals}. Following this we will identify which of these trees correspond to model structures on $[n]$. Let us begin by introducing our trees of interest:

\begin{defn}
 A \emph{(tri)colored tree} is a rooted planar tree such each each node may have edges coming from the set $\{blue,green,red\}$. We canonically order the branches coming out of a node so that blue is to the left of green and green is to the left of red. 
\end{defn}

\begin{remark}
Tricolored trees  are referred to as \emph{realizers of triangulations which are both minimal and maximal} in~\cite{intervals}. Further, in \cite{intervals} these trees are shown to be in bijection with \emph{stacked triangulations}. We have chosen instead to work with trees as it is the more natural approach to the results that we intend to prove.
\end{remark}

\begin{remark}\label{rem:tricoltrees}
If for each missing branch color we add a leaf and then forget all the colors, then we establish a bijection between tricolored trees and (uncolored) ternary trees. As such we can see the construction of tricolored trees as being iterative. Indeed a tricolored tree is either the empty tree, or a root with an ordered list of three (possibly empty) trees. We will return to this observation later.
\end{remark}

The goal of this section is to prove the following simple characterization in terms of colored trees of when a Kreweras pair defines a model structure.
	
	\begin{theorem}\label{mainthm}
		There is a bijection between colored trees and Kreweras pairs (to be explicitly constructed momentarily), under which a colored tree represents a model structure if and only if it does not have a red branch descended from any non-red branch.
	\end{theorem}
		
	We first construct a bijection between colored trees and Kreweras pairs $\RR\preccurlyeq \RR'$. We start by sorting the nodes of the tree left-to-right (via the standard inorder traversal) such that:
	\begin{enumerate}
		\item if $x\to y$ is a blue branch, then everything above $y$ (including $y$ itself) sits to the left of $x$;
		\item if $x\to y$ is a green branch, then everything above $y$ sits to the right of $x$;
		\item if $x\to y$ is a red branch, then everything above $y$ sits to the right of $x$ and also to the right of everything in the green branch of $x$.
	\end{enumerate}
	We say that such a tree is \emph{admissibly ordered}. Every tree has a unique admissible ordering. We refer the reader to \autoref{tree-example} for a tree ordered in such a fashion.
	
	Given some node $x$ in a tricolored tree $T$, the \emph{red-green component} of $x$ is defined to be the set of nodes $y$ for which the (unique) path connecting $x$ to $y$ consists of only red and/or green edges. We write $x\sim_T y$ if $y$ lies in the red-green component of $x$. This defines an equivalence relation on the vertices of $T$. We write $x\to_T y$ if $y$ is descended from $x$.
	
	\begin{construction}\label{cons:forward}
	Given an admissibly ordered tricolored tree $T$, we define $\pi_{\RR'}(x)$ to be the maximal $y\sim_T x$. We define $\pi_{\RR}(x)$ to be the maximal $y$ such that: 
	\begin{enumerate}
	\item $x\sim_T y$,
	\item $x\to_T y$, and 
	\item either $x=y$ or the path from $x$ to $y$ starts with a green edge. 
	\end{enumerate}
	Equivalently, if $x$ has no green children then we define $\pi_\RR(x) = x$; otherwise if $w$ is the green child of $x$ then $\pi_\RR(x)$ is the right-most descendant of $w$.
\end{construction}

\newpage
	\begin{example}\label{tree-example}
	We begin with the following tree that we have admissibly ordered:
	
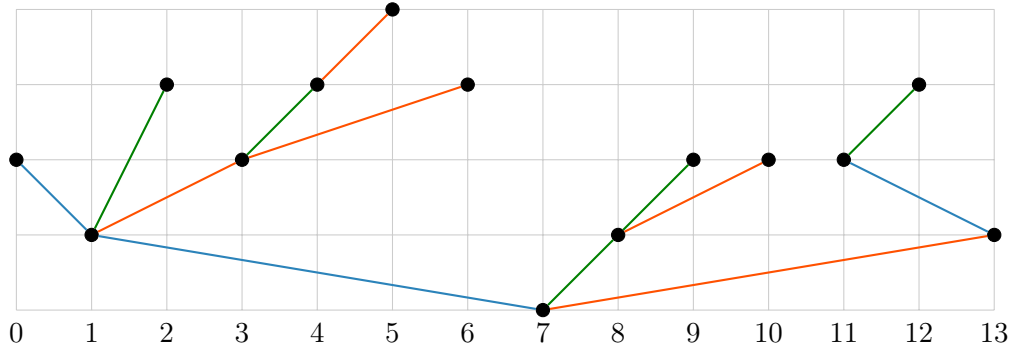
\begin{figure}[h]
		\centering
		\begin{tikzpicture}
		\draw[step=1.0,black!20,thin] (0.0,0.0) grid (13,4);
		\foreach \xtick in {0,...,13}  \node at (\xtick,-0.3) {\xtick};
		\node[circle, draw=black, fill=black, inner sep=0pt,minimum size=5pt] (0) at  (0, 2) {};
		\node[circle, draw=black, fill=black, inner sep=0pt,minimum size=5pt] (1) at  (1, 1) {};
		\node[circle, draw=black, fill=black, inner sep=0pt,minimum size=5pt] (2) at  (2, 3) {};
		\node[circle, draw=black, fill=black, inner sep=0pt,minimum size=5pt] (3) at  (3, 2) {};
		\node[circle, draw=black, fill=black, inner sep=0pt,minimum size=5pt] (4) at  (4, 3) {};
		\node[circle, draw=black, fill=black, inner sep=0pt,minimum size=5pt] (5) at  (5, 4) {};
		\node[circle, draw=black, fill=black, inner sep=0pt,minimum size=5pt] (6) at  (6, 3) {};
		\node[circle, draw=black, fill=black, inner sep=0pt,minimum size=5pt] (7) at  (7, 0) {};
		\node[circle, draw=black, fill=black, inner sep=0pt,minimum size=5pt] (8) at  (8, 1) {};
		\node[circle, draw=black, fill=black, inner sep=0pt,minimum size=5pt] (9) at  (9, 2) {};
		\node[circle, draw=black, fill=black, inner sep=0pt,minimum size=5pt] (10) at  (10, 2) {};
		\node[circle, draw=black, fill=black, inner sep=0pt,minimum size=5pt] (11) at  (11, 2) {};
		\node[circle, draw=black, fill=black, inner sep=0pt,minimum size=5pt] (12) at  (12, 3) {};
		\node[circle, draw=black, fill=black, inner sep=0pt,minimum size=5pt] (13) at  (13, 1) {};
		\draw[thick,gBlue] (0) -- (1);
		\draw[thick,Green] (1) -- (2);
		\draw[thick,gRed] (1) -- (3);
		\draw[thick,Green] (3) -- (4);
		\draw[thick,gRed] (4) -- (5);
		\draw[thick,gRed] (3) -- (6);
		\draw[thick,gBlue] (7) -- (1);
		\draw[thick,Green] (7) -- (8);
		\draw[thick,Green] (8) -- (9);
		\draw[thick,gRed] (8) -- (10);
		\draw[thick,gRed] (7) -- (13);
		\draw[thick,gBlue] (11) -- (13);
		\draw[thick,Green] (11) -- (12);
		\end{tikzpicture}
	\end{figure}
	
	We then read off the values of $\pi_{\RR'}(x)$ to be:
\begin{table}[h]
\begin{tabular}{c|c|c|c|c|c|c|c|c|c|c|c|c|c|c}
$x$             & 0 & 1 & 2 & 3 & 4 & 5 & 6 & 7 & 8 & 9 & 10 & 11 & 12 & 13 \\ \hline
$\pi_{\RR'}(x)$ & 0 & 6 & 6 & 6 & 6 & 6 & 6 & 13 & 13 & 13 & 13  & 12  & 12  & 13 
\end{tabular}
\end{table}
	
	and the values of $\pi_{\RR}(x)$ to be:
\begin{table}[h]
\begin{tabular}{c|c|c|c|c|c|c|c|c|c|c|c|c|c|c}
$x$             & 0 & 1 & 2 & 3 & 4 & 5 & 6 & 7 & 8 & 9 & 10 & 11 & 12 & 13 \\ \hline
$\pi_{\RR}(x)$ & 0 & 2 & 2 & 5 & 4 & 5 & 6 & 10 & 9 & 9 & 10  & 12  & 12  & 13 
\end{tabular}
\end{table}	
	\end{example}
	
	\begin{lemma}
		Let $T$ be an admissibly ordered tricolored tree. If $x\to_T y$ and $x\leqslant z\leqslant y$, then $x\to_T z$. Similarly if $x\to_T y$ and $y\leqslant z\leqslant x$, then $x\to_T z$.
	\end{lemma}
	\begin{proof}
  This follows from admissible orderings corresponding to inorder traversal; we leave the details to the reader.
		% We show the first statement; the proof of the second statement is analogous. Since we already know by assumption that $z>x$, showing $x\to_T z$ is equivalent to showing $z$ is descended from $x$. So let $x\wedge z$ be the lowest node on the unique minimal path connecting $x$ and $z$, and suppose $x\wedge z\neq x$. Let $B$ (resp., $G$, $R$) be the set of nodes descended from $x\wedge z$ along a path starting with a blue (resp., green, red) edge. Since $y$ is descended from $x$, $y$ and $x$ lie in the same set, say $G$ for concreteness. Since $z\leqslant y$ we cannot have $z\in R$ since all nodes in $R$ lie to the right of all nodes in $G$. Similarly since $z\geqslant x$ we can't have $z\in B\cup\{x\wedge z\}$ since all nodes in $B$ lie to the left of $x\wedge z$ which lies to the left of all nodes in $G$. Thus $z\in G$; but this contradicts the assumption that $x\wedge z$ lies on the minimal path connecting $x$ and $z$.
	\end{proof}
	
	\begin{theorem}
		The above construction is a bijection from the set of tricolored trees with $n+1$ nodes to the set of composition closed premodel structures on $[n]$.
	\end{theorem}

	\begin{proof}
		We will begin by showing that
		\begin{enumerate} 
		\item \autoref{cons:forward} always outputs a composition closed premodel structure, and
		\item \autoref{cons:forward}  is injective.
		\end{enumerate}
		After this, we will use the inductive process described in \autoref{rem:tricoltrees} to construct an inverse to \autoref{cons:forward}.
		
		Let $\pi_{\RR},\pi_{\RR'}$ be obtained from some tricolored tree $T$. By construction $\pi_{\RR}(x)$ and $\pi_{\RR'}(x)$ are non-decreasing and idempotent.
		
		Suppose $x\leqslant y\leqslant \pi_{\RR'}(x)$. We need to show $\pi_{\RR'}(y)\leqslant\pi_{\RR'}(x)$. Let $z$ be the lowest node in the red-green path connecting $x$ to $\pi_{\RR'}(x)$. Clearly $z\leqslant x$ and $\pi_{\RR'}(z) = \pi_{\RR'}(x)$, so we might as well assume $x=z$, so that $x\to_T\pi_{\RR'}(x)$. By the lemma this implies $y$ is also descended from $x$. If the path connecting $y$ to $\pi_{\RR'}(y)$ passes through $x$ then we can restrict this to give a red-green path from $x$ to $\pi_{\RR'}(y)$ which implies $\pi_{\RR'}(y)\leqslant\pi_{\RR'}(x)$ by maximality of $\pi_{\RR'}(x)$. Otherwise the path between $y$ and $\pi_{\RR'}(y)$ must remain entirely above $x$, which implies $\pi_{\RR'}(y)$ is also descended from $x$. Let $w$ be the right-most descendant of $x$, so that $\pi_{\RR'}(y)\leqslant w$ by definition. By construction of the admissible ordering, the path from $x$ to $w$ must be a red-green path, and hence $w\leqslant \pi_{\RR'}(x)$.
		
		Now suppose $x\leqslant y\leqslant \pi_{\RR}(x)$. Then $x\to_T\pi_{\RR}(x)$ by definition, so by the lemma this implies $y$ is also descended from $x$. If $y = x$ there's nothing to show. If $y$ were descended along blue edge then $y<x$ is a contradiction, and if $y$ were descended along a red edge then $y>\pi_\RR(x)$ is a contradiction. Thus $y$ is descended along a green edge. Let $x\to w$ be the green edge descended from $x$, so that $y$ and $\pi_{\RR}(x)$ are both descended from $w$. Then by construction, $\pi_\RR(x)$ is the rightmost descendant of $w$, and hence in particular $\pi_\RR(x)\geqslant\pi_\RR(y)$ since $\pi_\RR(y)$ is descended from $y$ which is descended from $w$.
		
		Thus $\RR,\RR'$ are transfer systems. It remains to show $\RR\preccurlyeq\RR'$, i.e., $\pi_{\RR}(x) = \pi_{\RR}(y)$ implies $\pi_{\RR'}(x) = \pi_{\RR'}(y)$. But $\pi_{\RR}(x) = \pi_{\RR}(y)$ implies in particular that there is a red-green path connecting $x$ and $y$, so this is immediate.
		
		We now prove (2). Let $T$ be an admissibly-ordered tree on the vertex set $[n]$, let $\RR\preccurlyeq\RR'$ be the corresponding composition closed pair, and let $x\in [n]$.
		
		First observe that we can recover $\sim_T$, since $\pi_{\RR'}(x) = \pi_{\RR'}(y)$ if and only if $x$ and $y$ lie in the same red-green component. Let $y>x$ be minimal such that $x\sim_T y$ and $x \,\RR \, y$. If $x$ has no green child then by definition of $\RR$ we must have $x = y$. Otherwise letting $w$ be the green child of $x$, by minimality of $y$ we have $x<y\leqslant w$. In the latter case, by the lemma, $y$ must be descended from $x$ along a green edge, and hence $y$ is descended from $w$. If $y<w$ then $y$ must be descended from $w$ along a blue edge, but this contradicts the assumption that $x$ and $y$ lie in the same red-green component. Thus $y=w$ is the green child of $x$, showing that we can recover all the green edges of $T$ from $\RR,\RR'$.
		
		Now let $y>x$ be minimal such that $x\sim_T y$ and $x \, \cancel\RR \, y$. Let $x\wedge y$ be the lowest point on the minimal path between $x$ and $y$. We claim $y$ is the red child of $x\wedge y$. Suppose first that $x\wedge y = x$. Then $y$ is descended from $x$ and $x\sim_T y$, so if $y$ were descended along a green edge then we would have $x \, \RR \, y$, a contradiction. Thus $y$ is descended along a red edge, and in particular $x$ has a red child $w$. By the same argument as before, $y = w$ by minimality. In the general case, note that neither $x$ nor $y$ can be descended from $x\wedge y$ along a blue edge since $x\sim_T y$. Since $x<y$ we thus must have $x$ descended from $x\wedge y$ along an initially green path and $y$ descended from $x\wedge y$ along an initially red path. Another minimality argument shows $y$ must be the red child of $x\wedge y$.
		
		By induction we can assume that for all $x'<x$ we've identified the red child of $x'$ using only $\RR$ and $\RR'$. If no $y$ exists as above then we know $x$ has no red children. So suppose we've found some minimal $y$ as above. Then we know $y$ is the red child of $x\wedge y$, and the path from $x\wedge y$ to $x$ is red-green. Since red-green edges always point to the right this implies $x\wedge y\leqslant x$, so either $y$ is the red child of $x$ or there exists some $x\wedge y<x$ such that $y$ is the red child of $x\wedge y$. We can recognize the latter case by the induction hypothesis, so this shows we can identify all the red edges in $T$.
		
		All that remains is to show how to recover the blue edges from $\RR$ and $\RR'$. Let $T'$ be the red-green forest that we've constructed, and let $x$ be the root of some tree in $T'$. We claim either $x$ is the root of $T$ or $\pi_{\RR'}(x)+1$ is the blue parent of $x$. Indeed this is forced by admissibility of the ordering. Suppose $x$ is not the root of $T$ and let $y$ be the blue parent of $x$. Since $x$ is not the red or green descendant of any node, the node $\pi_{\RR'}(x)$ is the furthest-right descendant of $x$ and hence $\pi_{\RR'}(x)<y$ by admissibility. But if $y>\pi_{\RR'}(x)+1>x$ then by the lemma $\pi_{\RR'}(x)+1$ must be descended from $y$ along a blue edge which would imply $\pi_{\RR'}(x)+1$ is descended from $x$, contradicting maximality of $\pi_{\RR'}(x)$. Thus $y = \pi_{\RR'}(x)+1$ as claimed and we are done with showing points (1) and (2).
		
		We now construct an inverse to \autoref{cons:forward}. Suppose that we have a Kreweras pair $\RR\preccurlyeq \RR'$ coming from some tricolored tree. We wish to reconstruct the tricolored tree corresponding to it. As in \autoref{rem:tricoltrees} we can use the inductive nature of building tricolored trees to do this. Given the interval $\RR\preccurlyeq \RR'$ we wish to decompose it as a root, and an ordered list of three intervals which we label  blue, red, and green. This will, by induction, provide the required inverse map.
		
		The root of the tree $r$, will be given as the minimal element in $\pi^{-1}_{\RR'}(n)$. The labels for the blue subtree will be the interval $[1,r-1]$. For the green subtree, we consider the interval $[r+1,m]$ where $m$ is the maximal element in the block of $r$ in the non-crossing partition $\pi_\RR$. If $m=r$ then this green subtree is empty. The labels of the red subtree is then the remaining interval $[m+1,n]$. We can then obtain intervals from these subtrees by applying \autoref{cons:forward} to them. That is, we have divided the Kreweras pair $\RR\preccurlyeq \RR'$ into a root, and an ordered 3-tuple of Kreweras pairs for smaller $n$. By induction, on the number of vertices we see that this allows to build a map from Kreweras pairs to tricolored trees which is injective as required.
	\end{proof}
	
	Now that we have proved that colored trees are in bijection with premodel structures on $[n]$, we move towards isolating the collection of model structures among them, and provide a proof of \autoref{mainthm}. 
	
	\begin{prop}
		A colored tree has a red branch sitting somewhere above a blue branch (i.e., the source of the red branch is a descendant of the top of the blue branch) if and only if there exists some $x$ such that $\pi_{\RR}(x)<\pi_{\RR'}(x)<n$.
	\end{prop}
	\begin{proof}
		Suppose we have a red branch $x\to y$. Then $y$ must lie to the right of every descendant of the (possibly non-existent) green branch coming out of $x$, and hence we have $y>\pi_{\RR}(x)$. But clearly $\pi_{\RR'}(x)\geqslant y$, so this shows $\pi_{\RR}(x)<\pi_{\RR'}(x)$. As a partial converse, if $\pi_{\RR}(x)<\pi_{\RR'}(x)$ then either $x$ has a red branch or there must be an upward-directed green path from some $z$ to $x$ such that $z$ has a red branch.
		
		Now suppose $x$ sits somewhere above a blue branch $y\to z$ (i.e., $x$ is a descendant of $z$ in the tree order). Then all descendants of $z$ must lie to the left of $y$, and hence if there exists a red-green path from $x$ to some $w$ then $w<y$. Thus in particular $\pi_{\RR'}(x)<n$. Conversely, if all ancestor branches of $x$ are green or red then we have a green-red path to the root, and clearly there is a red-green path from the root $n$, so in this case $\pi_{\RR'}(x) = n$.
		
		Combining these observations gives the desired equivalence.
	\end{proof}
		
	We shall write $x \, \LL'\to\RR \, y$ to denote that there exists some $z$ with $x \, \LL' \, z \, \RR \, y$.
		
	\begin{prop}\label{prop1}
		Let $\RR\preccurlyeq \RR'$ be a Kreweras pair of transfer systems on $[n]$. Then $\pi_{\RR}(x)<\pi_{\RR'}(x)$ implies $\pi_{\RR'}(x)=n$ if and only if $x \, \LL'\to\RR \, y$, $z \, \LL'\to\RR \, y$, and $x\leqslant z$ implies $x \, \LL'\to\RR \, z$.
	\end{prop}
	\begin{proof}
		We begin with the contrapositive of the backwards implication. Suppose $\pi_{\RR}(x)<\pi_{\RR'}(x)<n$. Since $\pi_{\RR'}(x)<n$, we know $x\,\LL'\,y$ where $y=\pi_{\RR'}(x)+1$. Let $z = \pi_{\RR'}(x)$. Then we know $z\,\LL'\,y$. But if $x \, \LL' \, w \, \RR \, z$, then $z<y$ implies $w=x$ and hence $x \, \RR \, z = \pi_{\RR'}(x) > \pi_{\RR}(x)$. This contradicts our initial supposition, so it is not the case that $x \, \LL'\to\RR \, y$, $z \, \LL'\to\RR \, y$, and $x\leqslant z$ implies $x \, \LL'\to\RR \, z$ holds.
		
		Conversely, suppose $x \, \LL' \, x' \, \RR \, y$ and $z \, \LL' \, z' \, \RR \, y$ with $x\leqslant z$, but there does not exist any $w$ with $x \, \LL' \, w \, \RR \, z$. We can suppose $z<x'$ since otherwise $x'\leqslant z\leqslant y$ and $x' \, \RR \, y$ would imply $x \, \LL' \, x' \, \RR \, z$. Now let $w\geqslant x$ be minimal such that $w \, \RR \, z$ and suppose we have a square
		$$\begin{tikzcd}
		x\arrow[r]\arrow[d]&u\arrow[d,two heads,"\RR'"]\\
		w\arrow[r]& v
		\end{tikzcd}$$
		If $u<w$ then this forces $u \, \RR' \, w$, but $u \, \cancel\RR  \, w$ by transitivity of $\RR$ and minimality of $w$. Thus $\pi_{\RR}(u)<\pi_{\RR'}(u)$. On the other hand, $u<w\leqslant z<x'$, so if $\pi_{\RR'}(u) = n$ then we would have a diagram
		$$\begin{tikzcd}
		x\arrow[r]\arrow[d,swap,hook,"\LL'"]&u\arrow[d,two heads,"\RR'"]\\
		x'\arrow[r]& n
		\end{tikzcd}$$
		with no lift, a contradiction.
	\end{proof}
	
	Let $\theta_{\RR}(z)=y$ where $y\leqslant z$ is minimal such that $y \, \LL \, z$. The following lemma records the properties satisfied by this function that are dual to properties of $\pi$.
	
	\begin{lemma}\label{lemma1}
		Let $\RR$ be a transfer system on $[n]$. Then
		\begin{enumerate}
			\item $\theta_{\RR}$ is a non-increasing idempotent function, and $\theta_{\RR}(i)\leqslant j\leqslant i$ implies $\theta_{\RR}(j)\geqslant \theta_{\RR}(i)$.
			\item $j \, \LL \, i$ if and only if $\theta_{\RR}(i)\leqslant j\leqslant i$.
			\item For all $i\in[n]$, $\theta_{\RR}(i) = \max\{j<i \mid j \, \RR \, i\}+1$.
		\end{enumerate} 
	\end{lemma}
	
	One can observe that we have a red branch above a green branch somewhere if and only if there exist $x<y<z$ such that $x \, \RR \, z$ and $y \, \RR' \, z$ but $w \, \cancel \RR \,z$ for all $y\leqslant w<z$. But $x \, \RR \, z$ is equivalent to saying $x<\theta_{\RR}(z)$, and likewise $y \, \RR' \, z$ is equivalent to $y<\theta_{\RR'}(z)$. Furthermore $w \, \cancel \RR \, z$ for all $y\leqslant w<z$ is equivalent to $\theta_{\RR}(z)\leqslant y$. Thus we have shown
	
	\begin{prop}
		A colored tree associated with a transfer system $\RR$ has a red branch sitting somewhere above a green branch if and only if there exists some $z$ such that $0<\theta_{\RR}(z)<\theta_{\RR'}(z)$.
	\end{prop}
	
	Dualizing \autoref{prop1}, we obtain the following result.
	
	\begin{prop}\label{prop2}
		Let $\RR\preccurlyeq \RR'$ be a Kreweras pair. Then $\theta_{\RR}(x)<\theta_{\RR'}(x)$ implies $\theta_{\RR}(x)=0$ if and only if $x \, \LL'\to\RR \,y$, $x\,\LL'\to\RR \,z$, and $z\leqslant y$ implies $z\,\LL'\to\RR \,y$.
	\end{prop}
	
	Combining the above results leads immediately to a proof of \autoref{mainthm} as required.\hfill\qedsymbol

\begin{example}
The following diagrams we undergo the procedure of moving between a colored tree (\autoref{fig:ncexample}) to its Kreweras intervals of non-crossing partitions (\autoref{fig:ncpartitions}), to the transfer systems (\autoref{fig:transon6}) and finally to the model structure (\autoref{fig:modelon6}) on the poset $[6]$:

\begin{figure}[h]
		\centering
		\begin{tikzpicture}[yscale = 0.7]
		\draw[step=1.0,black!20,thin] (0.0,0.0) grid (6,3);
		\foreach \xtick in {0,...,6}  \node at (\xtick,-0.3) {\xtick};
		\node[circle, draw=black, fill=black, inner sep=0pt,minimum size=5pt] (0) at  (0, 1) {};
		\node[circle, draw=black, fill=black, inner sep=0pt,minimum size=5pt] (1) at  (1, 0) {};
		\node[circle, draw=black, fill=black, inner sep=0pt,minimum size=5pt] (2) at  (2, 2) {};
		\node[circle, draw=black, fill=black, inner sep=0pt,minimum size=5pt] (3) at  (3, 1) {};
		\node[circle, draw=black, fill=black, inner sep=0pt,minimum size=5pt] (4) at  (4, 1) {};
		\node[circle, draw=black, fill=black, inner sep=0pt,minimum size=5pt] (5) at  (5, 3) {};
		\node[circle, draw=black, fill=black, inner sep=0pt,minimum size=5pt] (6) at  (6, 2) {};
		\draw[thick,gBlue] (0) -- (1);
		\draw[thick,Green] (1) -- (3);
		\draw[thick,gRed] (1) -- (4);
		\draw[thick,gBlue] (3) -- (2);
		\draw[thick,gRed] (4) -- (6);
		\draw[thick,gBlue] (6) -- (5);
		\end{tikzpicture}
		\caption{An admissibly ordered colored tree.}\label{fig:ncexample}
	\end{figure}
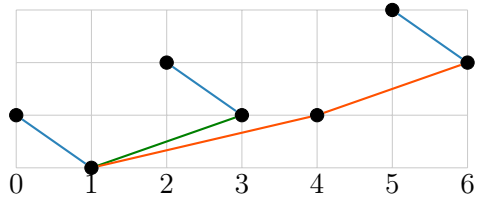

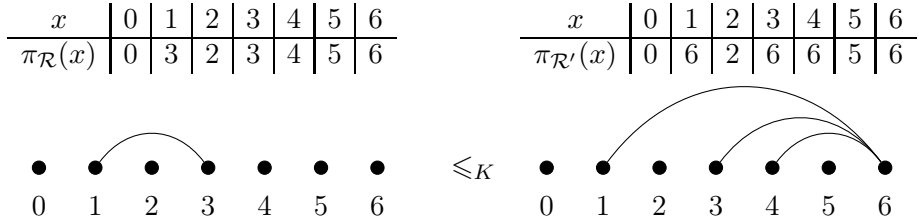
\begin{figure}[h]
\centering
\begin{tabular}{c|c|c|c|c|c|c|c}
$x$             & 0 & 1 & 2 & 3 & 4 & 5 & 6 \\ \hline
$\pi_{\RR}(x)$ & 0 & 3 & 2 & 3 & 4 & 5 & 6
\end{tabular}
\qquad \qquad
\begin{tabular}{c|c|c|c|c|c|c|c}
$x$             & 0 & 1 & 2 & 3 & 4 & 5 & 6 \\ \hline
$\pi_{\RR'}(x)$ & 0 & 6 & 2 & 6 & 6 & 5 & 6 
\end{tabular}

\vspace{-2mm}

\begin{tikzpicture}[xscale = 0.75]
\node[circle, draw=black, fill=black, inner sep=0pt,minimum size=5pt] (0) at  (0, 0) {};
\node[circle, draw=black, fill=black, inner sep=0pt,minimum size=5pt] (1) at  (1, 0) {};
\node[circle, draw=black, fill=black, inner sep=0pt,minimum size=5pt] (2) at  (2, 0) {};
\node[circle, draw=black, fill=black, inner sep=0pt,minimum size=5pt] (3) at  (3, 0) {};
\node[circle, draw=black, fill=black, inner sep=0pt,minimum size=5pt] (4) at  (4, 0) {};
\node[circle, draw=black, fill=black, inner sep=0pt,minimum size=5pt] (5) at  (5, 0) {};
\node[circle, draw=black, fill=black, inner sep=0pt,minimum size=5pt] (6) at  (6, 0) {};
\node at  (0, -0.5) {0};
\node at  (1, -0.5) {1};
\node at  (2, -0.5) {2};
\node at  (3, -0.5) {3};
\node at  (4, -0.5) {4};
\node at  (5, -0.5) {5};
\node at  (6, -0.5) {6};
\draw (1) to [bend left=45] (6);
\draw (3) to [bend left=45] (6);
\draw (4) to [bend left=45] (6);
\node at  (-1.3, 0) {$\leqslant_K$};
\begin{scope}[xshift = -9cm]
\node[circle, draw=black, fill=black, inner sep=0pt,minimum size=5pt] (0) at  (0, 0) {};
\node[circle, draw=black, fill=black, inner sep=0pt,minimum size=5pt] (1) at  (1, 0) {};
\node[circle, draw=black, fill=black, inner sep=0pt,minimum size=5pt] (2) at  (2, 0) {};
\node[circle, draw=black, fill=black, inner sep=0pt,minimum size=5pt] (3) at  (3, 0) {};
\node[circle, draw=black, fill=black, inner sep=0pt,minimum size=5pt] (4) at  (4, 0) {};
\node[circle, draw=black, fill=black, inner sep=0pt,minimum size=5pt] (5) at  (5, 0) {};
\node[circle, draw=black, fill=black, inner sep=0pt,minimum size=5pt] (6) at  (6, 0) {};
\node at  (0, -0.5) {0};
\node at  (1, -0.5) {1};
\node at  (2, -0.5) {2};
\node at  (3, -0.5) {3};
\node at  (4, -0.5) {4};
\node at  (5, -0.5) {5};
\node at  (6, -0.5) {6};
\draw (1) to [bend left=45] (3);
%\draw (4) to [bend left=45] (6);
\end{scope}
\end{tikzpicture}
\caption{noncrossing partitions on $[6]$ associated with \autoref{fig:ncexample}}\label{fig:ncpartitions}
\end{figure}

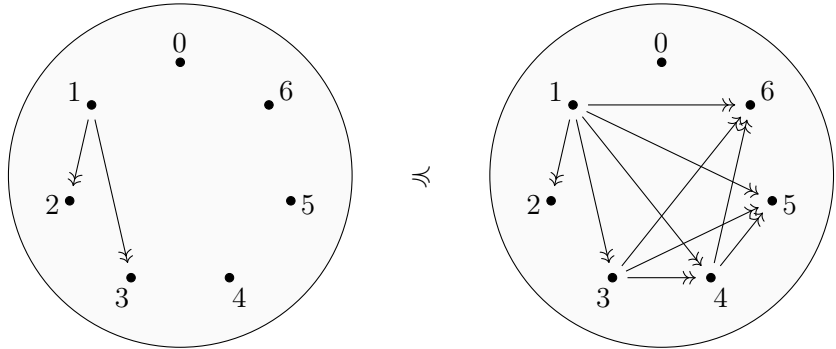
\begin{figure}[H]
\centering
\begin{tikzpicture}[scale = 0.8]
\node[circle, draw=black, fill=black!2, inner sep=0pt,minimum size=130pt] (0) at  (0, 0) {};
\node (pol) [draw=none, minimum size=3cm, regular polygon, regular polygon sides=7] at (0,0) {};
\foreach \n [count=\nu from 0, remember=\n as \lastn, evaluate={\nu+\lastn}] in {1,2,...,7}
\node[anchor=\n*(360/7)-140] at (pol.corner \n){\nu};
\foreach \n in {1,2,...,7}
\draw[fill = black] (pol.corner \n) circle (2pt);
\draw[->>, shorten >=2mm, shorten <=2mm](pol.corner 2) -- (pol.corner 3);
\draw[->>, shorten >=2mm, shorten <=2mm](pol.corner 2) -- (pol.corner 4);
\begin{scope}[xshift = 8cm]
\node[circle, draw=black, fill=black!2, inner sep=0pt,minimum size=130pt] (0) at  (0, 0) {};
\node (pol) [draw=none, minimum size=3cm, regular polygon, regular polygon sides=7] at (0,0) {};
\foreach \n [count=\nu from 0, remember=\n as \lastn, evaluate={\nu+\lastn}] in {1,2,...,7}
\node[anchor=\n*(360/7)-140] at (pol.corner \n){\nu};
\foreach \n in {1,2,...,7}
\draw[fill = black] (pol.corner \n) circle (2pt);
\draw[->>, shorten >=2mm, shorten <=2mm](pol.corner 2) -- (pol.corner 3);
\draw[->>, shorten >=2mm, shorten <=2mm](pol.corner 2) -- (pol.corner 4);
\draw[->>, shorten >=2mm, shorten <=2mm](pol.corner 2) -- (pol.corner 5);
\draw[->>, shorten >=2mm, shorten <=2mm](pol.corner 2) -- (pol.corner 6);
\draw[->>, shorten >=2mm, shorten <=2mm](pol.corner 2) -- (pol.corner 7);
\draw[->>, shorten >=2mm, shorten <=2mm](pol.corner 4) -- (pol.corner 7);
\draw[->>, shorten >=2mm, shorten <=2mm](pol.corner 4) -- (pol.corner 6);
\draw[->>, shorten >=2mm, shorten <=2mm](pol.corner 4) -- (pol.corner 5);
\draw[->>, shorten >=2mm, shorten <=2mm](pol.corner 5) -- (pol.corner 6);
\draw[->>, shorten >=2mm, shorten <=2mm](pol.corner 5) -- (pol.corner 7);
\end{scope}
\node at  (4, 0) {$\preccurlyeq$};
\end{tikzpicture}
\caption{The transfer systems associated to \autoref{fig:ncexample}.}\label{fig:transon6}
\end{figure}

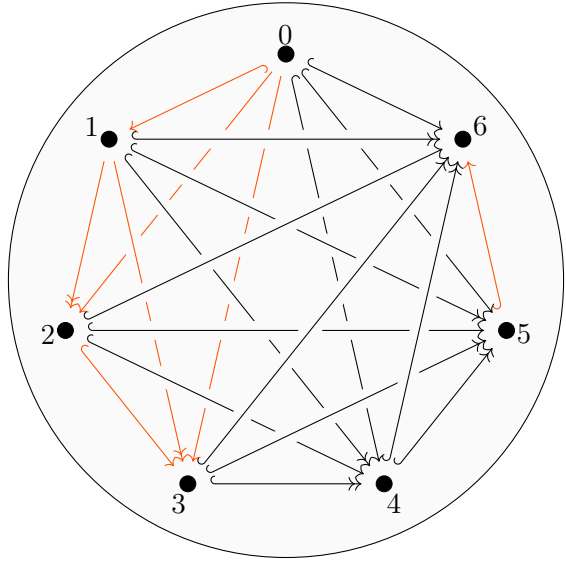
\begin{figure}[H]
\centering
\begin{tikzpicture}
\node[circle, draw=black, fill=black!2, inner sep=0pt,minimum size=210pt] (0) at  (0, 0) {};
\node (pol) [draw=none, minimum size=6cm, regular polygon, regular polygon sides=7] at (0,0) {};
\foreach \n [count=\nu from 0, remember=\n as \lastn, evaluate={\nu+\lastn}] in {1,2,...,7}
\node[anchor=\n*(360/7)-140] at (pol.corner \n){\nu};
\foreach \n in {1,2,...,7}
\draw[fill = black] (pol.corner \n) circle (3pt);
\draw[black!2, line width=2.5mm,shorten >=5mm, shorten <=5mm] (pol.corner 1) -- (pol.corner 2);
\draw[right hook->,gRed, shorten >=3mm, shorten <=3mm](pol.corner 1) -- (pol.corner 2);
\draw[black!2, line width=2.5mm,shorten >=5mm, shorten <=5mm] (pol.corner 1) -- (pol.corner 3);
\draw[->,gRed, shorten >=3mm, shorten <=3mm](pol.corner 1) -- (pol.corner 3);
\draw[black!2, line width=2.5mm,shorten >=5mm, shorten <=5mm] (pol.corner 2) -- (pol.corner 3);
\draw[->>,gRed, shorten >=3mm, shorten <=3mm](pol.corner 2) -- (pol.corner 3);
\draw[black!2, line width=2.5mm,shorten >=5mm, shorten <=5mm] (pol.corner 1) -- (pol.corner 4);
\draw[->,gRed, shorten >=3mm, shorten <=3mm](pol.corner 1) -- (pol.corner 4);
\draw[black!2, line width=2.5mm,shorten >=5mm, shorten <=5mm] (pol.corner 2) -- (pol.corner 4);
\draw[->>,gRed, shorten >=3mm, shorten <=3mm](pol.corner 2) -- (pol.corner 4);
\draw[black!2, line width=2.5mm,shorten >=5mm, shorten <=5mm] (pol.corner 3) -- (pol.corner 4);
\draw[right hook->,gRed, shorten >=3mm, shorten <=3mm](pol.corner 3) -- (pol.corner 4);
\draw[black!2, line width=2.5mm,shorten >=5mm, shorten <=5mm] (pol.corner 1) -- (pol.corner 5);
\draw[right hook->, shorten >=3mm, shorten <=3mm](pol.corner 1) -- (pol.corner 5);
\draw[black!2, line width=2.5mm,shorten >=5mm, shorten <=5mm] (pol.corner 2) -- (pol.corner 5);
\draw[right hook->>, shorten >=3mm, shorten <=3mm](pol.corner 2) -- (pol.corner 5);
\draw[black!2, line width=2.5mm,shorten >=5mm, shorten <=5mm] (pol.corner 3) -- (pol.corner 5);
\draw[right hook->, shorten >=3mm, shorten <=3mm](pol.corner 3) -- (pol.corner 5);
\draw[black!2, line width=2.5mm,shorten >=5mm, shorten <=5mm] (pol.corner 4) -- (pol.corner 5);
\draw[right hook->>, shorten >=3mm, shorten <=3mm](pol.corner 4) -- (pol.corner 5);
\draw[black!2, line width=2.5mm,shorten >=5mm, shorten <=5mm] (pol.corner 1) -- (pol.corner 6);
\draw[right hook->, shorten >=3mm, shorten <=3mm](pol.corner 1) -- (pol.corner 6);
\draw[black!2, line width=2.5mm,shorten >=5mm, shorten <=5mm] (pol.corner 2) -- (pol.corner 6);
\draw[right hook->>, shorten >=3mm, shorten <=3mm](pol.corner 2) -- (pol.corner 6);
\draw[black!2, line width=2.5mm,shorten >=5mm, shorten <=5mm] (pol.corner 3) -- (pol.corner 6);
\draw[right hook->, shorten >=3mm, shorten <=3mm](pol.corner 3) -- (pol.corner 6);
\draw[black!2, line width=2.5mm,shorten >=5mm, shorten <=5mm] (pol.corner 4) -- (pol.corner 6);
\draw[right hook->>, shorten >=3mm, shorten <=3mm](pol.corner 4) -- (pol.corner 6);
\draw[black!2, line width=2.5mm,shorten >=5mm, shorten <=5mm] (pol.corner 5) -- (pol.corner 6);
\draw[right hook->>, shorten >=3mm, shorten <=3mm](pol.corner 5) -- (pol.corner 6);
\draw[black!2, line width=2.5mm,shorten >=5mm, shorten <=5mm] (pol.corner 1) -- (pol.corner 7);
\draw[right hook->, shorten >=3mm, shorten <=3mm](pol.corner 1) -- (pol.corner 7);
\draw[black!2, line width=2.5mm,shorten >=5mm, shorten <=5mm] (pol.corner 2) -- (pol.corner 7);
\draw[right hook->>, shorten >=3mm, shorten <=3mm](pol.corner 2) -- (pol.corner 7);
\draw[black!2, line width=2.5mm,shorten >=5mm, shorten <=5mm] (pol.corner 3) -- (pol.corner 7);
\draw[right hook->, shorten >=3mm, shorten <=3mm](pol.corner 3) -- (pol.corner 7);
\draw[black!2, line width=2.5mm,shorten >=5mm, shorten <=5mm] (pol.corner 4) -- (pol.corner 7);
\draw[right hook->>, shorten >=3mm, shorten <=3mm](pol.corner 4) -- (pol.corner 7);
\draw[black!2, line width=2.5mm,shorten >=5mm, shorten <=5mm] (pol.corner 5) -- (pol.corner 7);
\draw[right hook->>, shorten >=3mm, shorten <=3mm](pol.corner 5) -- (pol.corner 7);
\draw[black!2, line width=2.5mm,shorten >=5mm, shorten <=5mm] (pol.corner 6) -- (pol.corner 7);
\draw[right hook->,gRed, shorten >=3mm, shorten <=3mm](pol.corner 6) -- (pol.corner 7);
\end{tikzpicture}
\caption{The model structure on $[6]$ encoded by \autoref{fig:ncexample}.}\label{fig:modelon6}
\end{figure}

\end{example}

	\subsection{Involutions}\label{sec2}
	
	The characterization of \autoref{mainthm} is --- by construction --- invariant under swapping blue and green,\footnote{On triangulations, this corresponds to reflection across the vertical axis of symmetry.} and hence this operation induces an involution on the collection of model structures. Presently, we characterize this involution explicitly. We begin with an intermediate lemma, referring the reader to~\cite[Definition 3.1]{boor} for the definition of interval partitions.
		
	\begin{lemma}\label{weak}
		Let $(\RR,\RR')$ be a model structure. Then $x$ and $y$ are weakly equivalent if and only if they lie in the same blue-green component of the colored tree corresponding to $(\RR,\RR')$.
	\end{lemma}
	
	\begin{proof}
		By \autoref{mainthm}, we know the colored tree corresponding to $(\RR,\RR')$ has a unique maximal red path $y_0\to y_1\to\cdots\to y_n$ where $y_0$ is the root, and such that no other branches are red. Given the method for which we order the tree nodes, one observes that the blue-green components partition $[n]$ into an interval partition $[n] = P_0\amalg P_1\amalg\cdots\amalg P_n$ where $y_i\in P_i$ and $x<y$ for all $x\in P_i$, $y\in P_j$, $i<j$. The claim is then that the $P_i$ are the weak equivalence classes of $(\RR,\RR')$.
		
		Let $x$ be the minimal element of $P_i$ and $y$ the maximal element. Then by the ordering that we use, it follows that the path from $y_i$ to $x$ is entirely blue and the path from $y_i$ to $y$ is entirely green. Thus $x \, \LL' \, y_i \, \RR \, y$. By the decomposition property of weak equivalences for model structures on a poset (c.f., \cite{boor}), this implies that all elements of $P_i$ are weakly equivalent.
		
		Now suppose $x\in P_i$ and $y\in P_j$ for some $i<j$. Then the path from $x$ to $y$ passes through the red path $y_i\to y_j$. Suppose $x \, \LL' \, z \, \RR \, y$. Then $y$ must be a descendant of $z$ along a green branch. But the first ancestors of $y$ are $y_0\to y_1\to\cdots \to y_j$, and since none of these branches are green we must have $z$ descended from $y_j$, which in particular forces $z\in P_j$. On the other hand, if $\tilde{x}$ is maximal in $P_i$ and $\tilde{y}$ is maximal in $P_j$ then one observes that $\tilde{x}\, \RR' \, \tilde{y}$, so if $x \, \LL' \,z$ then $z\leqslant \tilde{x}\in P_i$, a contradiction.
	\end{proof}
	
	\begin{remark}
		From the above lemma we can retrieve the enumeration result of \cite[Theorem 4.10]{boor} regarding model structures on $[n]$. Indeed, observe that a colored tree corresponding to a model structure is necessarily made up from a collection of blue-green trees growing up from a red field. Since blue-green trees correspond to binary trees, it follows that such trees correspond to ordered collections of binary trees whose node counts sum to $n$. Since binary trees are counted by Catalan numbers, this exactly recovers the enumeration of model structures by partitioning $[n-1]$ into weak equivalence classes and counting transfer systems on each part as in~\cite{boor}.
	\end{remark}
	
	We are now ready to give an explicit description of the blue--green involution on colored trees with respect to model structures.
	
		\begin{prop}\label{color-swap}
		Let $(\RR,\RR')$ be a model structure on $[n]$, and let $P$ be the interval partition generated by the weak equivalence classes. Then $(\RR,\RR')$ restricts to a model structure $(\RR_S,\RR'_S)$ on each $S\in P$. The model structure constructed by swapping blue and green in the colored tree corresponding to $(\RR,\RR')$ is the unique model structure $(\tilde{\RR},\tilde{\RR}')$ with the same weak equivalences as $(\RR,\RR')$ but for which $\tilde{\RR}'_S = \LL_S^{\mathsf{op}}$ and $\tilde{\RR}_S = (\LL'_S)^{\mathsf{op}}$ for all $S\in P$.
	\end{prop}
		
	\begin{proof}
		By \autoref{weak} and the way we sort things, clearly the operation of swapping blue and green preserves each weak equivalence class and reverses the order of all the nodes in each class. In what follows whenever we use $<$ or $\leqslant$ without clarification we refer to the order relevant for $\tilde{\RR}$ and $\tilde{\RR}'$, i.e., the order post-reversal. Thus when we refer to the ``opposite order'' we really mean the original order.
		
		Suppose $x<y$ lie in some blue-green component $P_i$. Then $x \, \tilde{\RR} \,y$ if and only if $y$ is descended from $x$ along a path that begins with a blue branch (in the original colored tree corresponding to $(\RR,\RR')$). But this means in the opposite order that for all $y\leqslant z<x$, $z$ is also descended from $x$ along a blue branch and hence $z \, \cancel\RR' \,x$ for all such $z$, which implies $y \, \LL' \, x$. Conversely, if $y$ is not descended from $x$ along a path that begins with a blue branch then since $x<y$ there must be some $z$ such that $y$ is descended from $z$ along a blue branch and $x$ is descended from $z$ along a green branch; but then it's easy to see in the opposite order that $y<z<x$ and $z \, \RR' \, x$, and hence $y \, \cancel\LL' \, x$. Thus $x \, \tilde{\RR} \, y$ if and only if $y \, \LL' \, x$ as claimed. Since $P_i$ is a weak equivalence class, we have $\tilde{\RR} = \tilde{\RR}'$ on $P_i$, so by~\cite[Theorem 3.10]{boor} we're done.
	\end{proof}

\bibliography{quillen}\bibliographystyle{alpha}

\end{document}